\documentclass[letterpaper,11pt]{article}
\usepackage{fullpage}
\usepackage{amsmath,amssymb,amsthm,enumerate, hyperref} %eucal
\usepackage[noadjust]{cite}
\usepackage{tikz, tikz-cd}
\usetikzlibrary{positioning}
\usepackage{verbatim}
\usepackage[english]{babel}

\hypersetup{hidelinks} %Remove hyperlink box.

\theoremstyle{definition}
\newtheorem{definition}{Definition}[section]

\newtheorem{example}[definition]{Example}
\newtheorem{conjecture}[definition]{Conjecture}

\newtheorem{notation}[definition]{Notation}
\newtheorem{note}[definition]{Note}

\theoremstyle{plain}
\newtheorem{theorem}[definition]{Theorem}
\newtheorem{lemma}[definition]{Lemma}
\newtheorem{corollary}[definition]{Corollary}
\newtheorem{proposition}[definition]{Proposition}

\def\MatF{\operatorname{Mat}_{d+1}(\mathbb{F})}
\def\EndV{\operatorname{End}(V)}
\def\th{\theta}

\begin{document}

\title{\bf Circular Hessenberg Pairs}

\author {
Jae-Ho Lee\thanks{Department of Mathematics and Statistics, University of North Florida, Jacksonville, FL 32224, U.S.A. 
\newline E-mail: jaeho.lee@unf.edu
}}
%\\ (updated 7/22/2020)
%}
\date{}

\maketitle

\begin{abstract}
A square matrix is called Hessenberg whenever each entry below the subdiagonal is zero and each entry on the subdiagonal is nonzero. 
Let $M$ denote a Hessenberg matrix.
Then $M$ is called circular whenever the upper-right corner entry of $M$ is nonzero and every other entry above the superdiagonal is zero.
A circular Hessenberg pair consists of two diagonalizable linear maps on a nonzero finite-dimensional vector space, that each act on an eigenbasis of the other one in a circular Hessenberg fashion.
Let $A, A^*$ denote a circular Hessenberg pair.
We investigate six bases for the underlying vector space that we find attractive.
We display the transition matrices between certain pairs of bases among the six.
We also display the matrices that represent $A$ and $A^*$ with respect to the six bases.
We introduce a special type of circular Hessenberg pair, said to be recurrent.
We show that a circular Hessenberg pair $A, A^*$ is recurrent if and only if $A, A^*$ satisfy the tridiagonal relations.
For a circular Hessenberg pair, there is a related object called a circular Hessenberg system.
We classify up to isomorphism the recurrent circular Hessenberg systems.
To this end, we construct four families of recurrent circular Hessenberg systems.
We show that every recurrent circular Hessenberg system is isomorphic to a member of one of the four families.


%\hfil\break
\end{abstract}

\section{Introduction}
This paper is about a linear algebraic object called a circular Hessenberg pair.
Before we describe this object, we first give some motivation.
The concept of a Leonard pair was introduced by Terwilliger \cite{2001TerLAA}.
A Leonard pair consists of two diagonalizable linear maps on a nonzero finite-dimensional vector space, that each act on an eigenbasis of the other one in an irreducible tridiagonal fashion; see Definition \ref{Def:LP}. 
A Leonard pair satisfies two relations called the tridiagonal relations  \cite{2001TerLAA}; see Lemma \ref{lem:LP, TDrels} below.
Notable papers about Leonard pairs are \cite{2001TerLAA, 2002TerRMJ, 2004TerLAA, 2005TerJCAM, 2006TerLNM}.
There is a generalization of a Leonard pair called a tridiagonal pair.
The concept of a tridiagonal pair was introduced by Ito, Tanabe, and Terwilliger \cite{2001ITP}. 
A tridiagonal pair satisfies the tridiagonal relations \cite{2001ITP}.
Notable papers about tridiagonal pairs are \cite{2011INT, 2010IT, 2007IT,2009IT,2007IT(2),2001ITP}.
For both Leonard pairs and tridiagonal pairs, there are connections to combinatorics \cite{1984BI, 1992TerJAC}, representation theory \cite{2001ITP, 2007IT, 2007IT(2), 2009IT, 2014BC}, special functions \cite{2006TerLNM, 2006Baseilhac, 2017BGV}, and statistical mechanics \cite{2005BK,2010BS}.
In \cite{2009GodjLAA}, Godjali introduced the concept of a Hessenberg pair as a generalization of a tridiagonal pair. 
In \cite{2010GodjLAA}, he considered a type of Hessenberg pair said to be thin, and he classified the thin Hessenberg pairs up to isomorphism.
In \cite{2017BGV}, Baseilhac, Gainutdinov, and Vu introduced the concept of a cyclic tridiagonal pair as another generalization of a tridiagonal pair.
They used this concept to study a higher-order generalization of the Onsager algebra.

In the present paper, we introduce the concept of a circular Hessenberg pair.
This concept is a special case of both a thin Hessenberg pair and a cyclic tridiagonal pair; see Note \ref{CHPvsCTP}.
Roughly speaking, a circular Hessenberg pair consists of two diagonalizable linear maps on a nonzero finite-dimensional vector space, that each act on an eigenbasis of the other one in a circular Hessenberg fashion; see Definition \ref{Def:CHP}.
We now summarize our results.
Let $A, A^*$ denote a circular Hessenberg pair.
We discuss six bases for the underlying vector space that we find attractive.
We display the transition matrices between certain pairs of bases among the six.
We also display the matrices that represent $A$ and $A^*$ with respect to the six bases.
We introduce a special type of circular Hessenberg pair, said to be recurrent; see Definition \ref{def recurrent CHP}.
We show that a circular Hessenberg pair $A, A^*$ is recurrent if and only if $A, A^*$ satisfy the tridiagonal relations; see Proposition \ref{rec <=> TD rels} and the sentence beneath it.
For a circular Hessenberg pair, there is a related object called a circular Hessenberg system; see Definition \ref{Def:CHS}.
We classify up to isomorphism the recurrent circular Hessenberg systems.
To do this, we give a method for constructing a recurrent circular Hessenberg system; see Theorem \ref{mainthm}.
Using this method, we obtain four families of recurrent circular Hessenberg systems; see Examples \ref{ex:beta not 2, -2}--\ref{ex:beta = 0}.
We prove that every recurrent circular Hessenberg system is isomorphic to a member of one of the four families; see Theorem \ref{thm:four families}.
We conjecture that every circular Hessenberg system is recurrent.

This paper is organized as follows.
In Section \ref{sec:CHpairs}, we introduce the concept of a circular Hessenberg pair and a  circular Hessenberg system. 
We also describe our main results.
In Section \ref{Sec:Phi-split seq}, we consider a circular Hessenberg system. 
We discuss the corresponding eigenvalue sequence, dual eigenvalue sequence, and split sequence.
In Section \ref{Sec:pf Thm(TD=>CHP)}, we prove Theorem \ref{mainthm}.
In Section \ref{sec:4 families}, we display four families of recurrent circular Hessenberg systems. 
We show that any recurrent circular Hessenberg system is isomorphic to a member of one of the four families.
In Section \ref{sec:SB}, we consider any circular Hessenberg pair $A, A^*$. 
We identify six bases for the underlying vector space that we find attractive.
We find the transition matrices between certain pairs of bases among the six.
We also display the matrices that represent $A$ and $A^*$ with respect to each basis.
In Section \ref{sec:scalar_xi}, we have some general comments about recurrent circular Hessenberg systems.
We end the paper with an appendix that contains some formulas involving recurrent sequences.

%%%%%%%%%%%%%%%%%%%%%%%%%%%%%%%%%%%%%%%%
%%%%%%%%%%%%%%%%%%%%%%%%%%%%%%%%%%%%%%%%
%%%%%%%%%%%%%%%%%%%%%%%%%%%%%%%%%%%%%%%%
\section{Circular Hessenberg pairs}\label{sec:CHpairs}
In this section, we introduce circular Hessenberg pairs and circular Hessenberg systems  along with a detailed motivation.
We use the following terms and notation.
Throughout this paper, $\mathbb{F}$ denotes a field.
All algebras and vector spaces discussed in this paper are over $\mathbb{F}$.
We fix an integer $d\geq 3$.
Let $V$ denote a vector space with dimension $d+1$.
Let $\EndV$ denote the algebra consisting of the $\mathbb{F}$-linear maps from $V$ to $V$.
Let $\MatF$ denote the algebra consisting of the $d+1$ by $d+1$ matrices that have entries in $\mathbb{F}$.
We index the rows and columns by $0,1,\ldots, d$.
Let $\mathbb{F}^{d+1}$ denote the vector space consisting of the column vectors of length $d+1$ that have entries in $\mathbb{F}$.
We view $\mathbb{F}^{d+1}$ as a left module for $\MatF$.
Let $\{v_i\}^d_{i=0}$ denote a basis for $V$.
For $B \in \EndV$ and $M\in \MatF$ we say $M$ \emph{represents} $B$ with respect to $\{v_i\}^d_{i=0}$ whenever $Bv_j = \sum^d_{i=0}M_{ij}v_i$ for $0 \leq j \leq d$.
There is an algebra isomorphism $\EndV \to \MatF$ that sends each $B\in \EndV$ to the unique matrix in $\MatF$ that represents $B$ with respect to $\{v_i\}^d_{i=0}$.
A matrix in $\MatF$ is called \emph{tridiagonal} whenever each nonzero entry lies on either the diagonal, the subdiagonal, or the superdiagonal. 
This tridiagonal matrix is called \emph{irreducible} whenever each entry on the subdiagonal is nonzero, and each entry on the superdiagonal is nonzero.
The following matrices are tridiagonal:
\begin{equation*}
	\begin{pmatrix}
	2 & 1 & 0 & 0 \\ 
	3 & 5 & 7 & 0 \\
	0 & 1 & 3 & 6 \\
	0 & 0 & 9 & 2
	\end{pmatrix},
	\qquad
	\begin{pmatrix}
	0 & 1 & 0 & 0 \\ 
	3 & 5 & 7 & 0 \\
	0 & 1 & 0 & 6 \\
	0 & 0 & 9 & 2
	\end{pmatrix},
	\qquad
	\begin{pmatrix}
	2 & 0 & 0 & 0 \\ 
	3 & 5 & 0 & 0 \\
	0 & 0 & 0 & 6 \\
	0 & 0 & 9 & 2
	\end{pmatrix}.
\end{equation*}
Observe that the tridiagonal matrices on the left and middle are irreducible.

\begin{definition}[cf. {\cite[Definition 1.1]{2001TerLAA}}]\label{Def:LP}
By a \emph{Leonard pair on $V$} we mean an ordered pair $A, A^*$ of elements in $\EndV$ such that:
\begin{itemize}
	\item[(i)] There exists a basis for $V$ with respect to which the matrix representing $A$ is diagonal and the matrix representing $A^*$ is irreducible tridiagonal.
	\item[(ii)] There exists a basis for $V$ with respect to which the matrix representing $A^*$ is diagonal and the matrix representing $A$ is irreducible tridiagonal.
\end{itemize} 
\end{definition}

\begin{note}\label{Note1}
It is a common notational convention to use $A^*$ to represent the conjugate-transpose of $A$. 
We are not using this convention. 
In a Leonard pair $A$, $A^*$ the linear maps $A$ and $A^*$ are arbitrary subject to (i) and (ii) of Definition \ref{Def:LP}.
\end{note}

We now recall the notion of a Leonard system.
An element $A \in \EndV$ is said to be \emph{diagonalizable} whenever $V$ is spanned by the eigenspaces of $A$.
The element $A$ is called \emph{multiplicity-free} whenever $A$ is diagonalizable and each eigenspace of $A$ has dimension one.
Note that $A$ is multiplicity-free if and only if $A$ has $d+1$ mutually distinct eigenvalues in $\mathbb{F}$.
Assume that $A$ is multiplicity-free.
Let $\{\theta_i\}^d_{i=0}$ denote an ordering of eigenvalues of $A$.
For $0 \leq i \leq d$, let $V_i$ denote the eigenspace of $A$ corresponding to $\theta_i$.
Further, define $E_i \in \EndV$ such that  $(E_i-I)V_i=0$ and $E_iV_j=0$ if $j\ne i$  $(0\leq j \leq d)$.
Here $I$ denotes the identity of $\EndV$.
We observe that (i) $E_iE_j = \delta_{ij}E_i$ $(0 \leq i, j \leq d)$; (ii) $AE_i = E_iA = \theta_i E_i$ $(0 \leq i \leq d)$; (iii) $\sum^d_{i=0} E_i = I$; (iv) $V_i=E_iV$ ($0\leq i \leq d$).
Moreover, 
\begin{equation}\label{prim idemp}
	E_i = \prod_{\substack{0 \leq j \leq d \\ j \ne i}}\frac{A-\theta_j I}{\theta_i-\theta_j } \qquad (0 \leq i \leq d).
\end{equation}
We call $E_i$ the \emph{primitive idempotent} of $A$ for $V_i$ (or $\theta_i$).

\begin{definition}[cf. {\cite[Definition 1.4]{2001TerLAA}}]\label{Def:LS}
By a \emph{Leonard system} on  $V$ we mean a sequence
\begin{equation*}\label{ALS}
	\Phi=(A; \{E_i\}^d_{i=0}; A^*; \{E^*_i\}^d_{i=0})
\end{equation*}
of elements in $\EndV$ that satisfy the conditions (i)--(v) below.
\begin{itemize}
	\item[(i)] Each of $A$, $A^*$ is multiplicity-free;
	\item[(ii)] $\{E_i\}^d_{i=0}$ is an ordering of the primitive idempotents of $A$;
	\item[(iii)] $\{E^*_i\}^d_{i=0}$ is an ordering of the primitive idempotents of $A^*$;
	\item[(iv)] $E_iA^*E_j = \begin{cases}
	0 & \text{if} \quad |i-j|>1; \\
	\ne 0 & \text{if} \quad |i-j|=1
	\end{cases}
	\qquad \qquad (0 \leq i, j \leq d)$;
	\item[(v)] $E^*_iAE^*_j = \begin{cases}
	0 & \text{if} \quad |i-j|>1; \\
	\ne 0 & \text{if} \quad |i-j|=1
	\end{cases}
	\qquad \qquad (0 \leq i, j \leq d)$.
\end{itemize}
We refer to $d$ as the \emph{diameter} of $\Phi$.
We say that $\Phi$ is \emph{over} $\mathbb{F}$.
\end{definition}

\noindent
The notion of isomorphism for Leonard systems was introduced in \cite[Definition 1.5]{2001TerLAA}.
Leonard systems are classified up to isomorphism \cite{2001TerLAA}.
This classification amounts to a linear algebraic version of Leonard's theorem \cite{1982Leonard}.

\begin{lemma}[cf. {\cite[Theorem 1.12]{2001TerLAA}}] \label{lem:LP, TDrels}
Let $A, A^*$ denote a Leonard pair on $V$.
Then there exists a sequence of scalars $\beta, \gamma, \gamma^*, \varrho, \varrho^*$ taken from $\mathbb{F}$ such that both 
\begin{align}
	0 &= [A, A^2A^* - \beta AA^*A + A^*A^2 - \gamma(AA^*+A^*A)-\varrho A^*], \label{TD1}\\
	0 &= [A^*, A^{*2}A - \beta A^*AA^* + AA^{*2} - \gamma^*(A^*A+AA^*)-\varrho^* A],\label{TD2}
\end{align}
where $[r,s]$ means $rs-sr$.
\end{lemma}
\noindent
The relations \eqref{TD1}, \eqref{TD2} are called the \emph{tridiagonal relations}.

\medskip
We now recall the notion of a Hessenberg pair.
This notion generalizes the notion of a Leonard pair.
A matrix in $\MatF$ is called \emph{Hessenberg} whenever each entry below the subdiagonal is zero and each entry on the subdiagonal is nonzero. 
The following matrices are Hessenberg:
\begin{equation*}
	\begin{pmatrix}
	2 & 1 & 4 & 9 \\ 
	3 & 5 & 7 & 8 \\
	0 & 1 & 3 & 6 \\
	0 & 0 & 9 & 2
	\end{pmatrix},
	\qquad
	\begin{pmatrix}
	2 & 1 & 0 & 0 \\ 
	3 & 5 & 7 & 0 \\
	0 & 1 & 0 & 6 \\
	0 & 0 & 9 & 2
	\end{pmatrix}, 
	\qquad
	\begin{pmatrix}
	0 & 0 & 0 & 0 \\ 
	3 & 0 & 0 & 0 \\
	0 & 1 & 0 & 0 \\
	0 & 0 & 9 & 0
	\end{pmatrix}.
\end{equation*}

\begin{definition}[cf. {\cite[Definition 1.1]{2010GodjLAA}}]\label{Def:HP}
By a \emph{Hessenberg pair on $V$} we mean an ordered pair $A, A^*$ of elements in $\EndV$ such that:
\begin{itemize}
	\item[(i)] There exists a basis for $V$ with respect to which the matrix representing $A$ is diagonal and the matrix representing $A^*$ is Hessenberg.
	\item[(ii)] There exists a basis for $V$ with respect to which the matrix representing $A^*$ is diagonal and the matrix representing $A$ is Hessenberg.
\end{itemize} 
\end{definition}

\begin{note}\label{note:THpair}
Our concept of a Hessenberg pair is slightly different from the one in \cite{2009GodjLAA, 2010GodjLAA}.
What we call a Hessenberg pair is called a \emph{thin} Hessenberg pair in \cite{2010GodjLAA}.
\end{note}

\noindent
For a Hessenberg pair $A, A^*$ on $V$, each of $A, A^*$ is multiplicity-free \cite[Lemma 2.1]{2010GodjLAA}.
We now recall a Hessenberg system.

\begin{definition}[cf. {\cite[Definition 2.2]{2010GodjLAA}}]\label{Def:HS}
By a \emph{Hessenberg system} on  $V$ we mean a sequence
\begin{equation*}\label{ALS}
	\Phi=(A; \{E_i\}^d_{i=0}; A^*; \{E^*_i\}^d_{i=0})
\end{equation*}
of elements in $\EndV$ that satisfy the conditions (i)--(v) below.
\begin{itemize}
	\item[(i)] Each of $A$, $A^*$ is multiplicity-free;
	\item[(ii)] $\{E_i\}^d_{i=0}$ is an ordering of the primitive idempotents of $A$;
	\item[(iii)] $\{E^*_i\}^d_{i=0}$ is an ordering of the primitive idempotents of $A^*$;
	\item[(iv)] $E_iA^*E_j = \begin{cases}
	0 & \text{if} \quad i-j>1; \\
	\ne 0 & \text{if} \quad i-j=1
	\end{cases}
	\qquad \qquad (0 \leq i, j \leq d)$;
	\item[(v)] $E^*_iAE^*_j = \begin{cases}
	0 & \text{if} \quad i-j>1; \\
	\ne 0 & \text{if} \quad i-j=1
	\end{cases}
	\qquad \qquad (0 \leq i, j \leq d)$.
\end{itemize}
We refer to $d$ as the \emph{diameter} of $\Phi$.
We say that $\Phi$ is \emph{over} $\mathbb{F}$.
\end{definition}

\begin{note}
What we call a Hessenberg system is called a \emph{thin} Hessenberg system in \cite{2010GodjLAA}.
\end{note}

\noindent
We remark that Hessenberg pairs do not satisfy the tridiagonal relations \eqref{TD1}, \eqref{TD2} in general.

\medskip
We now define a circular Hessenberg pair.
This is a special case of a Hessenberg pair.
Assume that $M\in\MatF$ is Hessenberg.
Then $M$ is called \emph{circular} whenever the $(0,d)$-entry of $M$ is nonzero and every other entry above the superdiagonal is zero.
The following matrices are circular Hessenberg:
\begin{equation*}
	\begin{pmatrix}
	2 & 1 & 0 & 9 \\ 
	3 & 5 & 7 & 0 \\
	0 & 1 & 3 & 6 \\
	0 & 0 & 9 & 2
	\end{pmatrix},
	\qquad
	\begin{pmatrix}
	2 & 0 & 0 & 9 \\ 
	3 & 5 & 7 & 0 \\
	0 & 1 & 0 & 6 \\
	0 & 0 & 9 & 2
	\end{pmatrix}, 
	\qquad
	\begin{pmatrix}
	0 & 0 & 0 & 9 \\ 
	3 & 0 & 0 & 0 \\
	0 & 1 & 0 & 0 \\
	0 & 0 & 9 & 0
	\end{pmatrix}.
\end{equation*}

\begin{definition}\label{Def:CHP}
By a \emph{circular Hessenberg pair} (or \emph{CH pair}) on $V$ we mean an ordered pair $A$, $A^*$ of elements in $\EndV$ such that:
\begin{itemize}
	\item[(i)] There exists a basis for $V$ with respect to which the matrix representing $A$ is diagonal and the matrix representing $A^*$ is circular Hessenberg.
	\item[(ii)] There exists a basis for $V$ with respect to which the matrix representing $A^*$ is diagonal and the matrix representing $A$ is circular Hessenberg.
\end{itemize}
\end{definition}

\begin{lemma}\label{lem:CHP->HP}
Let $A, A^*$ denote a CH pair on $V$.
Then the pair $A, A^*$ is Hessenberg. 
Moreover, each of $A, A^*$ is multiplicity-free.
\end{lemma}
\begin{proof}
The first assertion directly follows from the definition of a circular Hessenberg matrix.
The second assertion follows from the first assertion and the comment below Note \ref{note:THpair}.
\end{proof}

\noindent
We now introduce a \emph{circular Hessenberg system}.

\begin{definition} \label{Def:CHS}
By a \emph{circular Hessenberg system} (or \emph{CH system}) on  $V$ we mean a sequence
\begin{equation}\label{CHS}
	\Phi=(A; \{E_i\}^d_{i=0}; A^*; \{E^*_i\}^d_{i=0})
\end{equation}
of elements in $\EndV$ that satisfy the conditions (i)--(v) below.
\begin{itemize}
	\item[(i)] Each of $A$, $A^*$ is multiplicity-free;
	\item[(ii)] $\{E_i\}^d_{i=0}$ is an ordering of the primitive idempotents of $A$;
	\item[(iii)] $\{E^*_i\}^d_{i=0}$ is an ordering of the primitive idempotents of $A^*$;
	\item[(iv)] $E_iA^*E_j = \begin{cases}
	0 & \text{if} \quad 1<i-j \ \text{ or } \ 1<j-i<d; \\
	\ne 0 & \text{if} \quad 1=i-j \ \text{ or } \ j-i=d
	\end{cases}
	\qquad \qquad (0 \leq i, j \leq d)$;
	\item[(v)] $E^*_iAE^*_j = \begin{cases}
	0 & \text{if} \quad 1<i-j \ \text{ or } \ 1<j-i<d; \\
	\ne 0 & \text{if} \quad 1=i-j \ \text{ or } \ j-i=d
	\end{cases}
	\qquad \qquad (0 \leq i, j \leq d)$.
\end{itemize}
We refer to $d$ as the \emph{diameter} of $\Phi$. 
We say that $\Phi$ is \emph{over} $\mathbb{F}$.
\end{definition}

CH pairs and CH systems are related as follows.
Let $(A; \{E_i\}^d_{i=0}; A^*; \{E^*_i\}^d_{i=0})$ denote a CH system on $V$.
Then $A$, $A^*$ is a CH pair on $V$.
Conversely, let $A, A^*$ denote a CH pair on $V$. 
Then each of $A$, $A^*$ is multiplicity-free by Lemma \ref{lem:CHP->HP}.
Moreover, there exists an ordering $\{E_i\}^d_{i=0}$ of the primitive idempotents of $A$, and there exists an ordering $\{E^*_i\}^d_{i=0}$ of the primitive idempotents of $A^*$, such that $(A; \{E_i\}^d_{i=0}; A^*; \{E^*_i\}^d_{i=0})$ is a CH system on $V$.

\begin{note}\label{CHPvsCTP}
Let $A, A^*$ denote a CH pair on $V$.
By Lemma \ref{lem:CHP->HP}, $A, A^*$ is a Hessenberg pair on $V$.
We next explain why $A, A^*$ is a cyclic tridiagonal pair in the sense of \cite[Definition 1.1]{2017BGV}.
By the comment above the note, there is a CH system $\Phi=(A; \{E_i\}^d_{i=0}; A^*; \{E^*_i\}^d_{i=0})$ on $V$.
Define $N=d+1$ and consider the cyclic group $\mathbb{Z}_N=\mathbb{Z}/N\mathbb{Z}$.
The group elements are denoted by $0,1,\ldots, d$.
By Definition \ref{Def:CHS}(iv),(v) we have 
\begin{align*}
	A^*E_iV & \subseteq E_{i-1}V+E_iV+E_{i+1}V, \\
	AE^*_iV & \subseteq E^*_{i-1}V+E^*_iV+E^*_{i+1}V
\end{align*}
for $i\in \mathbb{Z}_N$.
Let $W$ denote a nonzero subspace of $V$ such that $AW\subseteq W$ and $A^*W\subseteq W$. 
We show that $W=V$.
Let $i \in \mathbb{Z}_N$.
By \eqref{prim idemp} and since $AW\subseteq W$, we have $E_iW\subseteq W$.
Moreover, $E_iW\ne 0$ if and only if $E_iW=E_iV$.
A similar comment applies to $E^*_i$.
Since $W\ne 0$, there exists $j\in \mathbb{Z}_N$ such that $E_jW\ne 0$. 
Therefore,  $E_jV\subseteq  W$.
By Definition \ref{Def:CHS}(iv) we have $E_{i+1}A^*E_i \ne 0$ for $i \in \mathbb{Z}_N$.
Taking $i=j$ we find that $E_{j+1} V \subseteq W$.
Repeating this argument, we have $E_r V \subseteq W$ for all $r \in\mathbb{Z}_N$.
Therefore, $W=V$.
By these comments, $A, A^*$ is a cyclic tridiagonal pair on $V$.
\end{note}

We will be discussing the notion of isomorphism for CH pairs and CH systems (cf. \cite[Definition 2.5]{2009GodjLAA}).
We now clarify what it means. 
Let $A, A^*$ denote a CH pair on $V$ and let $B, B^*$ denote a CH pair on $V'$.
By an \emph{isomorphism of CH pairs} from $A, A^*$ to $B, B^*$, we mean a vector space isomorphism $\sigma: V \to V'$ such that $\sigma A= B\sigma$ and $\sigma A^* = B^*\sigma$.
We say the CH pairs $A, A^*$ and $B, B^*$ are \emph{isomorphic} whenever there exists an isomorphism of CH pairs from $A, A^*$ and $B, B^*$.
Let $\sigma:V\to V'$ be an isomorphism of vector spaces.
For $X\in \EndV$ abbreviate $X^\sigma=\sigma X \sigma^{-1}$.
Observe that the map $\EndV \to \operatorname{End}(V')$, $X \mapsto X^\sigma$ is an isomorphism of algebras.
For a CH system $\Phi=(A; \{E_i\}^d_{i=0}; A^*; \{E^*_i\}^d_{i=0})$ on $V$, write $\Phi^\sigma:=(A^\sigma; \{E^\sigma_i\}^d_{i=0}; A^{*\sigma}; \{E^{*\sigma}_i\}^d_{i=0})$.
Observe that $\Phi^\sigma$ is a CH system on $V'$.
Let $\Phi'$ denote a CH system on $V'$.
By an \emph{isomorphism of CH systems} from $\Phi$ to $\Phi'$, we mean a vector space isomorphism $\sigma: V\to V'$ such that $\Phi^\sigma=\Phi'$. 
We say that the CH systems $\Phi$ and $\Phi'$ are \emph{isomorphic} whenever there exists an isomorphism of CH systems from $\Phi$ to $\Phi'$.

\medskip
Regarding Lemma \ref{lem:LP, TDrels}, it is natural to ask if a CH pair satisfies the tridiagonal relations \eqref{TD1}, \eqref{TD2}.

\begin{conjecture}\label{conj1}
Let $A, A^*$ denote a CH pair on $V$.
Then there exists a sequence of scalars $\beta, \gamma, \gamma^*, \varrho, \varrho^*$ taken from $\mathbb{F}$ such that both \eqref{TD1} and \eqref{TD2} hold.
\end{conjecture}

Shortly we will give another version of the above conjecture.
We have some comments about CH pairs and CH systems.
Let $\Phi=(A; \{E_i\}^d_{i=0}; A^*; \{E^*_i\}^d_{i=0})$ denote a CH system on $V$.
For $0\leq i \leq d$, let $\theta_i$ (resp. $\theta^*_i$) denote the eigenvalue of $A$ (resp. $A^*$) corresponding to $E_i$ (resp. $E^*_i$).
We refer to $\{\theta_i\}^d_{i=0}$ (resp. $\{\theta^*_i\}^d_{i=0}$) as the \emph{eigenvalue sequence} (resp. \emph{dual eigenvalue sequence}) of $\Phi$.
By \cite[Proposition 5.9]{2010GodjLAA} there exists a unique sequence $\{\phi_i\}^d_{i=1}$ of nonzero scalars in $\mathbb{F}$ with the following property: there exists a basis for $V$ with respect to which the matrices representing $A$ and $A^*$ are
\begin{equation}\label{CHP split mat}
	A: \quad \begin{pmatrix}
	\theta_d & & & & & \mathbf{0} \\
	1 & \theta_{d-1} \\
	 & 1 & \theta_{d-2} & \\
	 & & \cdot & \cdot \\
	 & & & \cdot & \cdot \\ 
	\mathbf{0} & & & & 1 & \theta_0
	\end{pmatrix},
	\qquad \qquad 
	A^*: \quad \begin{pmatrix}
	\theta^*_0 & \phi_1 & & & & \mathbf{0}\\
	 & \theta^*_1 & \phi_2 & \\
	 & & \theta^*_2 & \cdot \\
	 & & & \cdot & \cdot &\\
	 & & & & \cdot & \phi_d\\
	\mathbf{0} & & & & & \theta^*_d
	\end{pmatrix}.
\end{equation}
The sequence $\{\phi_i\}^d_{i=1}$ is called the \emph{$\Phi$-split sequence}. 
The above basis is called a \emph{$\Phi$-split basis}.
In Section \ref{Sec:Phi-split seq}, we shall discuss the $\Phi$-split sequence and the $\Phi$-split basis in detail.
For the rest of this paper we use the following notation.

\begin{notation}\label{notation:scalars}
Let $\{\theta_i\}^d_{i=0}$, $\{\theta^*_i\}^d_{i=0}$, $\{\phi_i\}^d_{i=1}$ denote scalars in $\mathbb{F}$ such that
\begin{itemize}
	\item[(i)] $\theta_i\ne \theta_j$, $\theta^*_i\ne \theta^*_j$ \quad if \quad $i \ne j$ \quad $(0 \leq i, j \leq d)$,
	\item[(ii)] $\phi_i \ne 0$ \quad  ($1\leq i \leq d$).
\end{itemize} 
\end{notation}

\begin{definition}\label{def:A A*}
Referring to Notation \ref{notation:scalars}, define the matrices $A, A^*\in \MatF$ by
\begin{equation*}\label{eq:matrices form}
	A =  \begin{pmatrix}
	\theta_d & & & & & \mathbf{0} \\
	1 & \theta_{d-1} \\
	 & 1 & \theta_{d-2} & \\
	 & & \cdot & \cdot \\
	 & & & \cdot & \cdot \\ 
	\mathbf{0} & & & & 1 & \theta_0
	\end{pmatrix},
	\qquad \qquad 
	A^* =  \begin{pmatrix}
	\theta^*_0 & \phi_1 & & & & \mathbf{0}\\
	 & \theta^*_1 & \phi_2 & \\
	 & & \theta^*_2 & \cdot \\
	 & & & \cdot & \cdot &\\
	 & & & & \cdot & \phi_d\\
	\mathbf{0} & & & & & \theta^*_d
	\end{pmatrix}.
\end{equation*}
Observe that $A$, $A^*$ are multiplicity-free. 
For $0\leq i \leq d$, let $E_i$ (resp. $E^*_i$) denote the primitive idempotent of $A$ (resp. $A^*$) with respect to $\theta_i$ (resp. $\theta^*_i$).
\end{definition}

Recall the scalars $\{\theta_i\}^d_{i=0}$, $\{\theta^*_i\}^d_{i=0}$, $\{\phi_i\}^d_{i=1}$ from Notation \ref{notation:scalars}.
In \cite[Section 11]{2001TerLAA}, the following scalars $\{\vartheta_i\}^{d+1}_{i=0}$ are introduced:
\begin{equation}\label{vartheta}
\begin{split}
	&\vartheta_i = \phi_i-(\theta^*_i-\theta^*_0)(\theta_{d-i+1}-\theta_0)
	\qquad \qquad (1\leq i \leq d), \\
	&\vartheta_0  =0, \qquad \quad \vartheta_{d+1}=0.
\end{split}
\end{equation}

\noindent
Next we discuss how the scalars $\{\vartheta_i\}^{d+1}_{i=0}$ are related to the tridiagonal relations \eqref{TD1}, \eqref{TD2}.
In what follows, we will use the concept of a $\beta$-recurrent sequence; see Definition \ref{def:beta-rec}.

\begin{lemma}[cf. {\cite[Theorem 12.5]{2001TerLAA}}]\label{lem:TD-beta rec}
Let  the scalars $\{\theta_i\}^d_{i=0}$, $\{\theta^*_i\}^d_{i=0}$, $\{\phi_i\}^d_{i=1}$ be as in Notation {\rm\ref{notation:scalars}}. 
Let the matrices $A$, $A^*$ be as in Definition {\rm\ref{def:A A*}}.
Let $\beta$ denote any scalar in $\mathbb{F}$. 
Then there exist scalars $\gamma, \gamma^*, \varrho, \varrho^*$ in $\mathbb{F}$ such that
\begin{align*}
	0 &= [A, A^2A^* - \beta AA^*A + A^*A^2 - \gamma(AA^*+A^*A)-\varrho A^*],\\
	0 &= [A^*, A^{*2}A - \beta A^*AA^* + AA^{*2} - \gamma^*(A^*A+AA^*)-\varrho^* A]
\end{align*}
if and only if {\rm(i)--(iii)} hold below.
\begin{enumerate}[\normalfont(i)]
	\item The sequence $\{\theta_i\}^d_{i=0}$ is $\beta$-recurrent.
	\item The sequence $\{\theta^*_i\}^d_{i=0}$ is $\beta$-recurrent.
	\item The sequence $\{\vartheta_i\}^{d+1}_{i=0}$ from \eqref{vartheta} is $\beta$-recurrent.
\end{enumerate}
\end{lemma}

We now state our first main theorem.

\begin{theorem}\label{mainthm}
Referring to Notation {\rm\ref{notation:scalars}} and Definition {\rm\ref{def:A A*}}, assume that there exists $\beta \in \mathbb{F}$ such that {\rm(i)} the sequence $\{\theta_i\}^d_{i=0}$ is $\beta$-recurrent; {\rm(ii)} the sequence $\{\theta^*_i\}^d_{i=0}$ is $\beta$-recurrent; {\rm(iii)} the sequence $\{\vartheta_i\}^{d+1}_{i=0}$ from \eqref{vartheta} is $\beta$-recurrent.
Further assume that $\vartheta_1\ne \vartheta_d$. 
Then the sequence 
\begin{equation}\label{seq CHS}
	\Phi=(A; \{E_i\}^d_{i=0}; A^*; \{E^*_i\}^d_{i=0})
\end{equation}
is a CH system on $\mathbb{F}^{d+1}$.
\end{theorem}
\noindent
The proof of Theorem \ref{mainthm} appears in Section \ref{Sec:pf Thm(TD=>CHP)}.

\medskip
Motivated by Theorem \ref{mainthm}, we make a definition.

\begin{definition}\label{def beta-rec}
Let $\Phi=(A; \{E_i\}^d_{i=0}; A^*; \{E^*_i\}^d_{i=0})$ denote a CH system on $V$, with eigenvalue sequence $\{\theta_i\}^d_{i=0}$ and dual eigenvalue sequence $\{\theta^*_i\}^d_{i=0}$.
Let $\{\phi_i\}^d_{i=1}$ denote the $\Phi$-split sequence.
Then for $\beta \in \mathbb{F}$, we say that $\Phi$ is \emph{$\beta$-recurrent} whenever
\begin{enumerate}[\normalfont(i)]
	\item the sequence $\{\theta_i\}^d_{i=0}$ is $\beta$-recurrent,
	\item the sequence $\{\theta^*_i\}^d_{i=0}$ is $\beta$-recurrent,
	\item the sequence $\{\vartheta_i\}^{d+1}_{i=0}$ from \eqref{vartheta} is $\beta$-recurrent.
\end{enumerate}
\end{definition}

\begin{definition}\label{def recurrent}
Let $\Phi=(A; \{E_i\}^d_{i=0}; A^*; \{E^*_i\}^d_{i=0})$ denote a CH system on $V$.
We say that $\Phi$ is \emph{recurrent} whenever there exists $\beta \in \mathbb{F}$ such that $\Phi$ is $\beta$-recurrent.
\end{definition}

\begin{definition}\label{def recurrent CHP}
Let $A, A^*$ denote a CH pair on $V$. 
We say that $A, A^*$ is \emph{recurrent} whenever there exists a CH system $\Phi=(A; \{E_i\}^d_{i=0}; A^*; \{E^*_i\}^d_{i=0})$ on $V$ that is recurrent.
\end{definition}

\noindent
By Definitions \ref{def beta-rec} and \ref{def recurrent}, we observe that the CH system from Theorem \ref{mainthm} is recurrent.
We now give another version of Conjecture \ref{conj1}.

\begin{conjecture}\label{conj2}
Every CH system on $V$ is recurrent.
\end{conjecture}

\noindent
We remark that Conjecture \ref{conj1} is equivalent to Conjecture \ref{conj2}.
Indeed, if Conjecture \ref{conj2} is true, then Conjecture \ref{conj1} is also true by Lemma \ref{lem:TD-beta rec}.
Conversely, suppose that Conjecture \ref{conj1} is true.
Let $\Phi=(A; \{E_i\}^d_{i=0}; A^*; \{E^*_i\}^d_{i=0})$ denote a CH system on $V$.
Then by Lemma \ref{lem:TD-beta rec}, $\Phi$ is $\beta$-recurrent, where $\beta$ is from Conjecture \ref{conj1}.
By Definition \ref{def recurrent}, $\Phi$ is recurrent.
Therefore, Conjecture \ref{conj2} is true.

%%%%%%%%%%%%%%%%%%%%%%%%%%%%%%%%%%%%%%%%
%%%%%%%%%%%%%%%%%%%%%%%%%%%%%%%%%%%%%%%%
%%%%%%%%%%%%%%%%%%%%%%%%%%%%%%%%%%%%%%%%
\section{The $\Phi$-split sequence and the $\Phi$-split basis}\label{Sec:Phi-split seq}

We continue to discuss the CH system $\Phi=(A; \{E_i\}^d_{i=0}; A^*; \{E^*_i\}^d_{i=0})$ from Definition \ref{Def:CHS}.
In the previous section, we discussed the $\Phi$-split sequence and the $\Phi$-split basis. 
In this section, we discuss these topics in more detail. 
Let $\{\theta_i\}^d_{i=0}$ (resp. $\{\theta^*_i\}^d_{i=0}$) denote the eigenvalue sequence (resp. dual eigenvalue sequence) of $\Phi$.
By a \emph{decomposition} of $V$, we mean a sequence $\{U_i\}^d_{i=0}$ of one-dimensional subspaces of $V$ such that
\begin{equation*}\label{def: decomp}
	V=\sum^d_{i=0}U_i \qquad \qquad  (\text{direct sum}).
\end{equation*}
For example, each of the sequences $\{E_iV\}^d_{i=0}$ and $\{E^*_iV\}^d_{i=0}$ is a decomposition of $V$.
We now recall the $\Phi$-split decomposition \cite[Section 4]{2010GodjLAA}.
For $0 \leq i \leq d$ define
\begin{equation*}\label{def:Ui}
	U_i = (E^*_0V+E^*_1V+ \cdots +E^*_i V) \cap (E_0V + E_1V + \cdots + E_{d-i}V).
\end{equation*}
Then the sequence $\{U_i\}^d_{i=0}$ is a decomposition of $V$.
Observe that
\begin{equation}\label{U0, Ud}
	U_0 = E^*_0V, \qquad \qquad U_d = E_0V.
\end{equation}
Moreover, 
\begin{align}
	\label{action A on Ui}
	&&(A-\theta_{d-i}I)U_i & = U_{i+1} \qquad (0 \leq i \leq d-1), && (A-\theta_0I)U_d = 0,&& \\
	\label{action A* on Ui}
	&&(A^*-\theta^*_{i}I)U_i & = U_{i-1} \qquad (1 \leq i \leq d), && (A^*-\theta^*_0I)U_0 = 0.&&
\end{align}
The sequence $\lbrace U_i \rbrace_{i=0}^d$ is called the \emph{$\Phi$-split decomposition of $V$}.
By \eqref{U0, Ud} and \eqref{action A on Ui},
\begin{equation}\label{space Ui}
	U_i = (A-\theta_{d-i+1}I) \cdots (A-\theta_{d-1}I)(A-\theta_{d}I) E^*_0V \qquad \quad (0\leq i \leq d).
\end{equation}
Combining \eqref{action A on Ui} and \eqref{action A* on Ui} we find that for $1\leq i \leq d$,
\begin{equation}\label{eig-space U(i)}
	(A-\theta_{d-i+1}I)(A^*-\theta^*_iI) U_{i} = U_{i}.
\end{equation}
Observe that $U_i$ is invariant under $(A-\theta_{d-i+1}I)(A^*-\theta^*_iI)$ and the corresponding eigenvalue is a nonzero element of $\mathbb{F}$.
We denote this eigenvalue by $\phi_i$.

\begin{lemma}\label{lem:phi-eigspace U(i-1)}
With the above notation, for $1\leq i \leq d$ the subspace $U_{i-1}$ is invariant under $(A^*-\theta^*_iI)(A-\theta_{d-i+1}I)$ and the corresponding eigenvalue is $\phi_i$.
\end{lemma}
\begin{proof}
Pick $0\ne u\in U_{i-1}$ and $0\ne v \in U_i$.
By \eqref{action A on Ui}, there exists $0\ne \lambda \in \mathbb{F}$ such that $(A-\theta_{d-i+1}I)u = \lambda v$.
By \eqref{action A* on Ui}, there exists $0\ne \mu \in \mathbb{F}$ such that $(A^*-\theta^*_iI)v=\mu u$.
By these comments, we have $(A-\theta_{d-i+1}I)(A^*-\theta^*_iI)v = \lambda \mu v$.
By the comment below \eqref{eig-space U(i)}, we find $\lambda\mu  = \phi_i$.
Observe that $(A^*-\theta^*_iI)(A-\theta_{d-i+1}I)u = \lambda \mu u = \phi_i u$.
The result follows.
\end{proof}

\noindent
Fix a nonzero vector  $u^* \in E^*_0V$.
For $0\leq i \leq d$ define
\begin{equation}\label{phi-split basis}
	v_i = (A-\theta_{d-i+1}I) \cdots (A-\theta_{d-1}I)(A-\theta_{d}I)u^*.
\end{equation}
Comparing \eqref{space Ui} and \eqref{phi-split basis} we see that $v_i$ is a nonzero element in $U_i$. 
Therefore, the vectors $\{v_i\}^d_{i=0}$ form a basis for $V$.

\begin{proposition}\label{prop: CHpair mat split}
Consider the basis $\{v_i\}^d_{i=0}$ for $V$ from \eqref{phi-split basis}.
With respect to this basis, the matrices representing $A$ and $A^*$ are
\begin{equation}\label{ALP split mat}
	A: \quad \begin{pmatrix}
	\theta_d & & & & & \mathbf{0} \\
	1 & \theta_{d-1} \\
	 & 1 & \theta_{d-2} & \\
	 & & \cdot & \cdot \\
	 & & & \cdot & \cdot \\ 
	\mathbf{0} & & & & 1 & \theta_0
	\end{pmatrix},
	\qquad \qquad 
	A^*: \quad \begin{pmatrix}
	\theta^*_0 & \phi_1 & & & & \mathbf{0}\\
	 & \theta^*_1 & \phi_2 & \\
	 & & \theta^*_2 & \cdot \\
	 & & & \cdot & \cdot &\\
	 & & & & \cdot & \phi_d\\
	\mathbf{0} & & & & & \theta^*_d
	\end{pmatrix}.
\end{equation}
\end{proposition}
\begin{proof}
Consider the action of $A$ on $\{v_i\}^d_{i=0}$.
By \eqref{phi-split basis}, for $0\leq i \leq d-1$ we have $(A-\theta_{d-i}I)v_i = v_{i+1}$, and therefore $Av_i = \theta_{d-i}v_i + v_{i+1}$.
By the equation on the right in \eqref{action A on Ui}, $v_d$ is an eigenvector of $A$ with eigenvalue $\theta_0$.
By these comments, the matrix on the left in \eqref{ALP split mat} represents $A$ with respect to $\{v_i\}^d_{i=0}$.
Next, we consider the action of $A^*$ on $\{v_i\}^d_{i=0}$.
By Lemma \ref{lem:phi-eigspace U(i-1)}, for $1 \leq i \leq d$ we have $(A^*-\theta^*_iI)v_i = (A^*-\theta^*_iI)(A-\theta_{d-i+1}I)v_{i-1} = \phi_iv_{i-1}$, and therefore $A^*v_i = \theta^*_iv_i + \phi_iv_{i-1}$.
By the equation on the right in \eqref{action A* on Ui}, $v_0$ is an eigenvector of $A^*$ with eigenvalue $\theta^*_0$.
By these comments, the matrix on the right in \eqref{ALP split mat} represents $A^*$ with respect to $\{v_i\}^d_{i=0}$.
\end{proof}
\noindent
We comment on Proposition \ref{prop: CHpair mat split}.
Comparing \eqref{CHP split mat} and \eqref{ALP split mat}, we see that $\{\phi_i\}^d_{i=1}$ is the $\Phi$-split sequence and  $\{v_i\}^d_{i=0}$ is a $\Phi$-split basis for $V$.

\begin{proposition}\label{rec <=> TD rels}
Let $\beta\in \mathbb{F}$ and let $\Phi=(A; \{E_i\}^d_{i=0}; A^*; \{E^*_i\}^d_{i=0})$ denote a CH system on $V$.
Then the following {\rm(i)}, {\rm(ii)} are equivalent:
\begin{enumerate}[\normalfont(i)]
	\item $\Phi$ is $\beta$-recurrent.
	\item There exists a sequence of scalars $\gamma, \gamma^*, \varrho, \varrho^*$ taken from $\mathbb{F}$ such that both \eqref{TD1} and \eqref{TD2} hold.
\end{enumerate}
\end{proposition}
\begin{proof}
By Lemma \ref{lem:TD-beta rec}, Definition \ref{def beta-rec} and Proposition \ref{prop: CHpair mat split}.
\end{proof}
\noindent
From Proposition \ref{rec <=> TD rels}, we find that a CH pair $A, A^*$ is recurrent if and only if $A, A^*$ satisfy the tridiagonal relations \eqref{TD1} and \eqref{TD2}.

\medskip
We now define the dual of a CH system.
\begin{definition}
Let $\Phi=(A; \{E_i\}^d_{i=0}; A^*; \{E^*_i\}^d_{i=0})$ denote a CH system on $V$.
Observe that $\Phi^*=(A^*; \{E^*_i\}^d_{i=0}; A; \{E_i\}^d_{i=0})$ is a CH system on $V$.
We call $\Phi^*$ the \emph{dual} of $\Phi$.
\end{definition}

\begin{lemma}
Let $\Phi=(A; \{E_i\}^d_{i=0}; A^*; \{E^*_i\}^d_{i=0})$ denote a CH system on $V$.
Let $\{U_i\}^d_{i=0}$ denote the $\Phi$-split decomposition of $V$.
Then $\{U_{d-i}\}^d_{i=0}$ is the $\Phi^*$-split decomposition of $V$.
Moreover, the eigenvalue sequences, the dual eigenvalue sequences, and the split sequences of $\Phi$ and $\Phi^*$ are related as follows.
$$
{\renewcommand{\arraystretch}{1.5}
\begin{tabular}{c|cccc}
	CH system & eigenvalue sequence & dual eigenvalue sequence & split sequence \\
	\hline
	$\Phi$ & $\{\theta_i\}^d_{i=0}$ & $\{\theta^*_i\}^d_{i=0}$ & $\{\phi_i\}^d_{i=1}$ \\
	$\Phi^*$ & $\{\theta^*_i\}^d_{i=0}$ & $\{\theta_i\}^d_{i=0}$ & $\{\phi_{d-i+1}\}^d_{i=1}$
\end{tabular}}
$$
\end{lemma}
\begin{proof}
Routine.
\end{proof}

\noindent
Earlier we defined the $\Phi$-split basis of $V$. 
We now describe the $\Phi^*$-split basis of $V$.
Fix a nonzero vector $u \in E_0V$.
For $0\leq i \leq d$ define
\begin{equation}\label{phi*-split basis}
	v^*_i = (A^*-\theta^*_{d-i+1}I) \cdots (A^*-\theta^*_{d-1}I)(A^*-\theta^*_{d}I)u.
\end{equation}
Then $\{v^*_i\}^d_{i=0}$ is the $\Phi^*$-split basis for $V$.
Applying Proposition \ref{prop: CHpair mat split} to $\Phi^*$, we find that the matrices representing $A$ and $A^*$ with respect to $\{v^*_i\}^d_{i=0}$ are
\begin{equation}\label{ALP dual split mat}
	A: \quad \begin{pmatrix}
	\theta_0 & \phi_d & & & & \mathbf{0}\\
	 & \theta_1 & \phi_{d-1} & \\
	 & & \theta_2 & \cdot \\
	 & & & \cdot & \cdot &\\
	 & & & & \cdot & \phi_1\\
	\mathbf{0} & & & & & \theta_d
	\end{pmatrix},
	\qquad \qquad 
	A^*: \quad \begin{pmatrix}
	\theta^*_d & & & & & \mathbf{0} \\
	1 & \theta^*_{d-1} \\
	 & 1 & \theta^*_{d-2} & \\
	 & & \cdot & \cdot \\
	 & & & \cdot & \cdot \\ 
	\mathbf{0} & & & & 1 & \theta^*_0
	\end{pmatrix}.
\end{equation}

Next, we discuss how the $\Phi$-split bases and the $\Phi^*$-split bases are related.

\begin{definition}\label{notation epsilon}
Let $\Phi=(A; \{E_i\}^d_{i=0}; A^*; \{E^*_i\}^d_{i=0})$ denote a CH system on $V$.
Recall the nonzero vector $u^*\in E^*_0V$ from above line \eqref{phi-split basis} and the nonzero vector $u \in E_0V$ from above line \eqref{phi*-split basis}.
Observe that $E_0u^* \in E_0V$ and $E^*_0u \in E^*_0V$.
Therefore, there exist scalars $\varepsilon, \varepsilon^* \in \mathbb{F}$ such that
\begin{equation*}\label{scalars epsilon}
	E_0u^*=\varepsilon u, \qquad \qquad
	E^*_0u=\varepsilon^* u^*.
\end{equation*}
\end{definition}

\begin{lemma}\label{lem:tr(E0E*0)}
With reference to Definition \rm{\ref{notation epsilon}}, $\mathrm{tr}(E_0E^*_0)  =  \varepsilon\varepsilon^*$.
\end{lemma}
\begin{proof}
Recall the equation $E_0u^* =\varepsilon u$.
Applying $E^*_0$ to both sides and using $E^*_0u=\varepsilon^* u^*$, we have $E^*_0E_0u^* =\varepsilon\varepsilon^* u^*$. 
Since $u^*=E^*_0u^*$, it follows that $E^*_0E_0E^*_0 u^*=  \varepsilon\varepsilon^* E^*_0u^*$.
Therefore, we have $E^*_0E_0E^*_0 =  \varepsilon\varepsilon^* E^*_0$.
Take the trace of both sides to get $\mathrm{tr}(E^*_0E_0E^*_0)  =  \varepsilon\varepsilon^* \mathrm{tr}(E^*_0)$.
Since $\mathrm{tr}(E^*_0E_0E^*_0)=\mathrm{tr}(E_0E^*_0E^*_0)=\mathrm{tr}(E_0E^*_0)$ and $\mathrm{tr}(E^*_0)=1$, the result follows.
\end{proof}

\begin{lemma}\label{scalar epsilon epsilon*} 
With reference to Definition \rm{\ref{notation epsilon}},
\begin{equation*}
	\varepsilon \varepsilon^* = \frac{\phi_1\phi_2 \cdots \phi_d}{(\theta_0 - \theta_1)(\theta_0 - \theta_2) \cdots (\theta_0 - \theta_d)(\theta^*_0 - \theta^*_1)(\theta^*_0 - \theta^*_2) \cdots (\theta^*_0 - \theta^*_d)}.
\end{equation*}
\end{lemma}
\begin{proof}
By Lemma \ref{lem:tr(E0E*0)} and \cite[Lemma 7.5, Lemma 7.6]{2010GodjLAA}.
\end{proof}
\noindent
The scalars $\varepsilon$, $\varepsilon^*$ are nonzero by Lemma \ref{scalar epsilon epsilon*}.

\begin{note}
In \cite[Section 7]{2010GodjLAA}, a nonzero scalar $\nu$ was introduced to study Hessenberg systems.
We have $\nu^{-1}=\varepsilon\varepsilon^*$.
\end{note}

\begin{proposition}\label{prop:v v*}
Recall the $\Phi$-split basis $\{v_i\}^d_{i=0}$ from \eqref{phi-split basis} and the $\Phi^*$-split basis $\{v^*_i\}^d_{i=0}$ from \eqref{phi*-split basis}.  
Then the following {\rm(i)}, {\rm(ii)} hold:
\begin{enumerate}[\normalfont(i)]
	\item $v_i = \varepsilon  \dfrac{(\theta_0-\theta_{1})(\theta_0-\theta_{2}) \cdots (\theta_0 - \theta_{d})}{\phi_{i+1}\phi_{i+2} \cdots \phi_{d}}v^*_{d-i}$  \qquad   $(0\leq i \leq d)$.
	\item $v^*_i = \varepsilon^*  \dfrac{(\theta^*_0-\theta^*_{1})(\theta^*_0-\theta^*_{2}) \cdots (\theta^*_0 - \theta^*_{d})}{\phi_{1}\phi_{2} \cdots \phi_{d-i}}v_{d-i}$  \qquad $(0\leq i \leq d)$.
\end{enumerate}
\end{proposition}
\begin{proof}
(i): By the matrix on the left in \eqref{ALP split mat}, we have $A^*v_i = \theta^*_iv_i + \phi_iv_{i-1}$ for $1\leq i \leq d$.
This implies that $v_{i-1} = (A^*-\theta^*_iI)v_i/\phi_i$.
By induction on $i$, we have
\begin{equation}\label{eq: v_i formula}
	v_i = \frac{(A^*-\theta^*_{i+1}I)(A^*-\theta^*_{i+2}I) \cdots (A^*-\theta^*_{d}I)v_d}{\phi_{i+1}\phi_{i+2} \cdots \phi_{d}} \qquad \qquad (0\leq i \leq d-1).
\end{equation}
Evaluating \eqref{phi-split basis} at $i=d$ and using \eqref{prim idemp}, we have
\begin{equation}\label{eq: v_d}
	v_d = (A-\theta_{1}I)(A-\theta_{2}I) \cdots (A-\theta_{d}I)u^*
	= (\theta_0-\theta_{1})(\theta_0-\theta_{2}) \cdots (\theta_0 - \theta_{d})E_0u^*.
\end{equation}
Eliminate $v_d$ in \eqref{eq: v_i formula} using \eqref{eq: v_d} along with $E_0u^*=\varepsilon u$, and simplify the result to get
\begin{equation}\label{eq: v_d(2)}
	v_i = \varepsilon  \frac{(\theta_0-\theta_{1})(\theta_0-\theta_{2}) \cdots (\theta_0 - \theta_{d})}{\phi_{i+1}\phi_{i+2} \cdots \phi_{d}}(A^*-\theta^*_{i+1}I)(A^*-\theta^*_{i+2}I) \cdots (A^*-\theta^*_{d}I)u
\end{equation}
for  $0\leq i \leq d-1$.
By \eqref{phi*-split basis} we note that $v^*_{d-i}=(A^*-\theta^*_{i+1}I)(A^*-\theta^*_{i+2}I) \cdots (A^*-\theta^*_{d}I)u$.
Hence, (i) follows.

\smallskip
\noindent
(ii): Similar to (i).
\end{proof}

We finish this section with a comment.
Let $\Phi$ denote a CH system on $V$.
By the \emph{parameter array} of $\Phi$, we mean the sequence  $(\{\theta_i\}^d_{i=0}, \{\theta^*_i\}^d_{i=0}, \{\phi_i\}^d_{i=1})$, where $\{\theta_i\}^d_{i=0}$ is the eigenvalue sequence of $\Phi$,  $\{\theta^*_i\}^d_{i=0}$ is the dual eigenvalue sequence of $\Phi$, and $\{\phi_i\}^d_{i=1}$ is the $\Phi$-split sequence.
\begin{lemma}\label{iso_CHsystem}
Let $\Phi$ and $\Phi'$ denote CH systems over $\mathbb{F}$.
Then $\Phi$ and $\Phi'$ are isomorphic if and only if they have the same parameter array.
\end{lemma}
\begin{proof}
By \cite[Theorem 6.3]{2010GodjLAA}.
\end{proof}

%%%%%%%%%%%%%%%%%%%%%%%%%%%%%%%%%%%%%%%%
%%%%%%%%%%%%%%%%%%%%%%%%%%%%%%%%%%%%%%%%
%%%%%%%%%%%%%%%%%%%%%%%%%%%%%%%%%%%%%%%%
\section{The proof of Theorem \ref{mainthm}}\label{Sec:pf Thm(TD=>CHP)}
In this section we prove Theorem \ref{mainthm}.
Throughout this section, we refer to Notation \ref{notation:scalars} and Definition \ref{def:A A*}.
%Let the scalars $\{\theta_i\}^d_{i=0}$, $\{\theta^*_i\}^d_{i=0}$, $\{\phi_i\}^d_{i=1}$ be as in Notation \ref{notation:scalars}.
%Let the sequence $\{\vartheta_i\}^{d+1}_{i=0}$ be as in  \eqref{vartheta}.
%Consider the matrices $A, A^*$ in \eqref{eq:matrices form}.
%Clearly, each of $A, A^*$ is multiplicity-free.
%For $0\leq i \leq d$, let $E_i$ (resp. $E^*_i$) denote the primitive idempotent of $A$ (resp. $A^*$) associated with $\theta_i$ (resp. $\theta^*_i$).
Shortly we will be referring to the results of \cite{2020TerNote}. 
We note that the element $E_i$ in \cite{2020TerNote} corresponds to the $E_{d-i}$ in the present paper.

\begin{lemma}\label{lem: vanishing prod terms(1)}
The following {\rm(i), (ii)} hold. For $0\leq i, j \leq d$, 
\begin{enumerate}[\normalfont(i)]
	\item $E_iA^*E_j=
	\begin{cases}
	0 &\text{if} \quad i-j>1, \\
	\ne 0 &\text{if} \quad i-j=1.
	\end{cases}$
	\item $E^*_iAE^*_j=
	\begin{cases}
	0 &\text{if} \quad i-j>1, \\
	\ne 0 &\text{if} \quad i-j=1.
	\end{cases}$
\end{enumerate}
\end{lemma}
\begin{proof}
(i): Use \cite[Proposition 7.6]{2020TerNote}, \cite[Lemma 7.7]{2020TerNote}.\\ 
(ii): By \cite[Proposition 7.6]{2020TerNote}.
\end{proof}

\begin{lemma}\label{E0 normalizing}
For $0\leq i \leq d$, both $E_iE^*_0 \ne 0$ and $E^*_iE_0 \ne 0$.
\end{lemma}
\begin{proof}
By Lemma \ref{lem: vanishing prod terms(1)}(ii) and \cite[Lemma 7.5]{2020TerNote}.
\end{proof}

\begin{lemma}
Both
\begin{align}
	\sum^{d}_{i=2} E_0A^*E_iE^*_0(\theta_i-\theta_1) & = E_0E^*_0(\vartheta_1-\vartheta_d),\label{wraparound(1)} \\
	\sum^{d}_{i=2} E^*_0AE^*_iE_0(\theta^*_1-\theta^*_i) & = E^*_0E_0(\vartheta_1-\vartheta_d). \label{wraparound(2)}
\end{align}
\end{lemma}
\begin{proof}
By \cite[Proposition 8.4]{2020TerNote} and \cite[Lemma 8.5]{2020TerNote}.
\end{proof}

\begin{lemma}\label{lem: vanishing prod terms(2)}
Assume that there exists $\beta \in \mathbb{F}$ such that each of $\{\theta_i\}^d_{i=0}$, $\{\theta^*_i\}^d_{i=0}$, $\{\vartheta_i\}^{d+1}_{i=0}$ is $\beta$-recurrent.
Then the following {\rm(i)--(iv)} hold. 
\begin{enumerate}[\normalfont(i)]
	\item $E_iA^*E_j=0$ \quad if \quad $1<j-i<d$ \quad $(0\leq i, j \leq d)$.
	\item $E^*_iAE^*_j=0$ \quad if \quad $1<j-i<d$ \quad  $(0\leq i, j \leq d)$.
	\item $E_0A^*E_d\ne0$ \quad if and only if \quad $\vartheta_1 \ne \vartheta_d$.
	\item $E^*_0AE^*_d \ne 0$ \quad if and only if \quad $\vartheta_1 \ne \vartheta_d$.
\end{enumerate}
\end{lemma}
\begin{proof}
(i): By a slight modification of the proof of line (60) in \cite[Section 17]{2020TerNote}.\\
(ii):  Similar to the proof of (i).\\
(iii): Recall the equation \eqref{wraparound(1)}.
By (i), we have $E_0A^*E_j=0$ for $1< j < d$.
Applying this to the equation \eqref{wraparound(1)}, we have the equation $E_0A^*E_dE^*_0(\theta_d-\theta_1)=E_0E^*_0(\vartheta_1-\vartheta_d)$.
Note that $\theta_d \ne \theta_1$. 
Also, $E_0E^*_0 \ne 0$ by Lemma \ref{E0 normalizing}.
Therefore, $\vartheta_1\ne \vartheta_d$ if and only if $E_0A^*E_dE^*_0\ne0$.
By Lemma \ref{E0 normalizing} and \cite[Proposition 6.4]{2020TerNote} with $X=E_0A^*$ and $i=d$, $E_0A^*E_dE^*_0\ne0$ if and only if $E_0A^*E_d\ne0$.\\
(iv): Similar to the proof of (iii).
\end{proof}

We are now ready to prove Theorem \ref{mainthm}.

\begin{proof}[Proof of Theorem \ref{mainthm}]
To prove that the sequence $\Phi=(A; \{E_i\}^d_{i=0}; A^*; \{E^*_i\}^d_{i=0})$ is a CH system, we show that $\Phi$ satisfies the conditions (i)--(v) of Definition \ref{Def:CHS}.
By construction, the conditions (i)--(iii) hold.
The condition (iv) follows from Lemma \ref{lem: vanishing prod terms(1)}(i) and  Lemma \ref{lem: vanishing prod terms(2)}(i),(iii). 
The condition (v) follows from Lemma \ref{lem: vanishing prod terms(1)}(ii) and  Lemma \ref{lem: vanishing prod terms(2)}(ii),(iv).
The result follows.
\end{proof}

\section{Four families of CH systems}\label{sec:4 families}

In this section, we classify up to isomorphism the recurrent CH systems.
To do this, we first display four families of recurrent CH systems.
We then show that every recurrent CH system is isomorphic to a member of one of the four families.

For the rest of this section, let $\{\theta_i\}^d_{i=0}$, $\{\theta^*_i\}^d_{i=0}$, $\{\phi_i\}^d_{i=1}$ denote scalars in $\mathbb{F}$.
Let $\overline{\mathbb{F}}$ denote the algebraic closure of $\mathbb{F}$.

\begin{example}\label{ex:beta not 2, -2}
Fix $0\ne q\in \overline{\mathbb{F}}$ such that $q+q^{-1} \in \mathbb{F}$.
Assume that $q^i \ne 1$ for $1\leq i \leq d$ and $q^{d+1}=1$.
Notice that $q\ne \pm1$ since $d\geq 3$.
Set $\beta=q+q^{-1}$.
Observe that $\beta\ne \pm2$.
Let $a,b,c,a^*,b^*,c^*$ be scalars taken from $\overline{\mathbb{F}}$.
Assume
\begin{align}
	\theta_i & = a + bq^i + cq^{-i}, \label{eq:theta(i)}\\
	\theta^*_i & = a^* + b^*q^i + c^*q^{-i} \label{eq:theta*(i)}
\end{align}
for $0\leq  i \leq d$.
Assume that $c\ne bq^i$ and $c^*\ne b^*q^i$ for $1 \leq i \leq 2d-1$.
By these assumptions, the scalars $\{\theta_i\}^d_{i=0}$ are mutually distinct and $\beta$-recurrent.
Similarly, the scalars $\{\theta^*_i\}^d_{i=0}$ are mutually distinct and $\beta$-recurrent.
Pick distinct $y,z\in \overline{\mathbb{F}}$.
Assume
\begin{equation}\label{eq:phi(i)}
	\phi_i  = (q^i-1)(y-zq^{-i}) + (q^i-1)(q^{-i}-1)(b-cq^{i})(b^*-c^*q^{-i}) 
\end{equation}
for $1\leq i \leq d$.
Assume that $y, z$ are chosen so that $\phi_i\ne 0$ for $1\leq i \leq d$.
From \eqref{eq:theta(i)}--\eqref{eq:phi(i)}, the scalars $\{\vartheta_i\}^{d+1}_{i=0}$ in \eqref{vartheta} satisfy
\begin{equation}\label{eq:varth fm case1}
	\vartheta_i = (q^i-1)(y-zq^{-i})
\end{equation}
for $1\leq i \leq d$ along with $\vartheta_0=0$ and $\vartheta_{d+1}=0$.
Observe that the sequence $\{\vartheta_i\}^{d+1}_{i=0}$ is $\beta$-recurrent.
Also, observe that $\vartheta_1-\vartheta_d=q(1-q^{d-1})(y-z)$, so $\vartheta_1 \ne \vartheta_d$.
We have shown that the scalars $\{\theta_i\}^d_{i=0}$, $\{\theta^*_i\}^d_{i=0}$, $\{\vartheta_i\}^{d+1}_{i=0}$ satisfy the conditions of Theorem \ref{mainthm}.
By this theorem, the sequence 
\begin{equation}\label{CHS_beta ne pm2,-2}
	\Phi=(A; \{E_i\}^d_{i=0}; A^*; \{E^*_i\}^d_{i=0})
\end{equation}
from Definition \ref{def:A A*} forms a CH system on $\mathbb{F}^{d+1}$. 
By construction, $\Phi$ is $\beta$-recurrent.
\end{example}

\begin{example}\label{ex:beta = 2}
Assume that $\operatorname{Char}(\mathbb{F})=d+1$.
Set $\beta=2$.
Let $a,b,c,a^*,b^*,c^*$ be scalars taken from $\mathbb{F}$.
Assume
\begin{align}
	 \theta_i & = a + bi + ci(i-1)/2, \label{eq:th cl form, be=2}\\
	\theta^*_i & = a^* + b^*i + c^*i(i-1)/2 \label{eq:th* cl form, be=2}
\end{align}
for $0\leq i \leq d$.
Assume that $2b \ne c(1-i)$ and $2b^* \ne c^*(1-i)$ for $1 \leq i \leq 2d-1$.
By these assumptions, the scalars $\{\theta_i\}^d_{i=0}$ are mutually distinct and $\beta$-recurrent.
Similarly, the scalars $\{\theta^*_i\}^d_{i=0}$ are mutually distinct and $\beta$-recurrent.
Pick $y,z \in \mathbb{F}$ such that $2y\ne z$. 
Assume 
\begin{equation}\label{eq:phi cl form, be=2}
	\phi_i  = i\big(y+z(i-1)/2\big) - i^2\big(b+c(d-i)/2\big)\big(b^*+c^*(i-1)/2\big)
\end{equation}
for $1 \leq i \leq d$.
Assume that $y,z$ are chosen so that $\phi_i\ne 0$ for $1\leq i \leq d$.
From \eqref{eq:th cl form, be=2}--\eqref{eq:phi cl form, be=2}, the scalars $\{\vartheta_i\}^{d+1}_{i=0}$ in \eqref{vartheta} satisfy
\begin{equation}\label{eq:varth cl form, be=2}
	\vartheta_i = yi + zi(i-1)/2
\end{equation}
for $1\leq i \leq d$ along with $\vartheta_0=0$ and $\vartheta_{d+1}=0$.
Observe that the sequence $\{\vartheta_i\}^{d+1}_{i=0}$ is $\beta$-recurrent.
Also, observe that $\vartheta_1-\vartheta_d=2y-z$, so $\vartheta_1 \ne \vartheta_d$.
We have shown that the scalars $\{\theta_i\}^d_{i=0}$, $\{\theta^*_i\}^d_{i=0}$, $\{\vartheta_i\}^{d+1}_{i=0}$ satisfy the conditions of Theorem \ref{mainthm}.
By this theorem, the sequence 
\begin{equation}\label{CHS_beta=2}
	\Phi=(A; \{E_i\}^d_{i=0}; A^*; \{E^*_i\}^d_{i=0})
\end{equation}
from Definition \ref{def:A A*} forms a CH system on $\mathbb{F}^{d+1}$. 
By construction, $\Phi$ is $\beta$-recurrent.
\end{example}

\begin{example}\label{ex:beta = -2}
Assume that  $d$ is odd and $d\geq 5$ and $\operatorname{Char}(\mathbb{F})=(d+1)/2$.
Set $\beta=-2$.
Let $a,b,c,a^*,b^*,c^*$ be scalars taken from $\mathbb{F}$.
Assume
\begin{align}
	\theta_i & = a + b(-1)^i + c{i}(-1)^i, \label{eq:th cl form, be=-2}\\
	\theta^*_i & = a^* + b^*(-1)^i + c^*{i}(-1)^i \label{eq:th* cl form, be=-2}
\end{align}
for $0\leq i \leq d$.
Assume that none of $b, b^*, c, c^*$ is zero. Also assume that $2b\ne -ic$ and $2b^*\ne -ic^*$ for $1\leq i \leq 2d-1$ with $i$ odd.
By these assumptions, the scalars $\{\theta_i\}^d_{i=0}$ are mutually distinct and $\beta$-recurrent.
Similarly, the scalars $\{\theta^*_i\}^d_{i=0}$ are mutually distinct and $\beta$-recurrent.
Pick $y,z \in \mathbb{F}$ such that $z\ne 0$.
Assume 
\begin{equation}\label{eq:phi cl form, be=-2}
	\phi_i  =  y\Big((-1)^i-1\Big)+zi(-1)^{i} 
		+ \Big(b\big((-1)^i-1\big)-ci(-1)^i\Big)\Big(b^*\big((-1)^i-1\big)+c^*i(-1)^i\Big)
\end{equation}
for $1\leq i \leq d$.
Assume that $y,z$ are chosen so that $\phi_i\ne 0$ for $1\leq i \leq d$.
From \eqref{eq:th cl form, be=-2}--\eqref{eq:phi cl form, be=-2}, the scalars $\{\vartheta_i\}^{d+1}_{i=0}$ in \eqref{vartheta} satisfy
\begin{equation}\label{eq:varth cl form, be=-2}
	\vartheta_i = y\big((-1)^i-1\big) + zi(-1)^i
\end{equation}
for $1\leq i \leq d$ along with $\vartheta_0=0$ and $\vartheta_{d+1}=0$.
Observe that the sequence $\{\vartheta_i\}^{d+1}_{i=0}$ is $\beta$-recurrent.
Also, observe that $\vartheta_1-\vartheta_d=z(d-1)$, so $\vartheta_1 \ne \vartheta_d$.
We have shown that the scalars $\{\theta_i\}^d_{i=0}$, $\{\theta^*_i\}^d_{i=0}$, $\{\vartheta_i\}^{d+1}_{i=0}$ satisfy the conditions of Theorem \ref{mainthm}.
By this theorem, the sequence 
\begin{equation}\label{CHS_beta=-2}
	\Phi=(A; \{E_i\}^d_{i=0}; A^*; \{E^*_i\}^d_{i=0})
\end{equation}
from Definition \ref{def:A A*} forms a CH system on $\mathbb{F}^{d+1}$. 
By construction, $\Phi$ is $\beta$-recurrent.
\end{example}

\begin{example}\label{ex:beta = 0}
Assume that $d=3$ and $\operatorname{Char}(\mathbb{F})=2$.
Set $\beta=0$.
Let $a,b,c,a^*,b^*,c^*$ be scalars taken from $\mathbb{F}$.
Assume
\begin{align}
	 \theta_i & = a+ bi + c\binom{i}{2}, \label{eq:th cl form, be=0}\\ 
	\theta^*_i & = a^*+ b^*i + c^*\binom{i}{2} \label{eq:th* cl form, be=0}
\end{align}
for $0\leq i \leq 3$ and where we interpret the binomial coefficient as follows:
	\begin{equation}\label{eq:bin char=2}
	\binom{n}{2} = \begin{cases}
	0 &  \text{if } \quad n=0 \text{ or } n=1 \quad (\mathrm{mod}\ 4),\\
	1 & \text{if } \quad  n=2 \text{ or } n=3 \quad (\mathrm{mod}\ 4).	
	\end{cases}
	\end{equation}
Assume that none of $b, b^*, c, c^*$ is zero. Also assume that $b \ne c$ and $b^* \ne c^*$.
By these assumptions, the scalars $\{\theta_i\}^3_{i=0}$ are mutually distinct and $\beta$-recurrent.
Similarly, the scalars $\{\theta^*_i\}^3_{i=0}$ are mutually distinct and $\beta$-recurrent.
Pick $y,z \in \mathbb{F}$ such that $z\ne 0$.
Assume
\begin{equation}\label{eq:phi cl form, be=0}
	\phi_i  =  yi + z\binom{i}{2} + \left(bi+c\binom{i+1}{2}\right)\left(b^*i+c^*\binom{i}{2}\right)
\end{equation}
for $1 \leq i \leq 3$.
Assume that $y,z$ are chosen so that $\phi_i\ne 0$ for $1\leq i \leq 3$.
The scalars $\{\vartheta_i\}^{4}_{i=0}$ from \eqref{vartheta} satisfy
\begin{equation}\label{eq:varth cl form, be=0}
	\vartheta_i = yi + z\binom{i}{2}
\end{equation}
for $1\leq i \leq 3$ along with $\vartheta_0=0$ and $\vartheta_{4}=0$.
Observe that the sequence $\{\vartheta_i\}^{4}_{i=0}$ is $\beta$-recurrent.
Also, observe that $\vartheta_1-\vartheta_3=z$, so $\vartheta_1 \ne \vartheta_3$.
We have shown that the scalars $\{\theta_i\}^3_{i=0}$, $\{\theta^*_i\}^3_{i=0}$, $\{\vartheta_i\}^{4}_{i=0}$ satisfy the conditions of Theorem \ref{mainthm}.
By this theorem, the sequence 
\begin{equation}\label{CHS_beta=0}
	\Phi=(A; \{E_i\}^3_{i=0}; A^*; \{E^*_i\}^3_{i=0})
\end{equation}
from Definition \ref{def:A A*} forms a CH system on $\mathbb{F}^{4}$. 
By construction, $\Phi$ is $\beta$-recurrent.
\end{example}

We have displayed four families of recurrent CH systems. 
Our next goal is to show that every recurrent CH system over $\mathbb{F}$ is isomorphic to a member of one of the four families. 
In order to obtain this result, we need a lemma.

\begin{lemma}\label{lem:vth1=vthd}
Let $\Phi=(A; \{E_i\}^d_{i=0}; A^*; \{E^*_i\}^d_{i=0})$ denote a recurrent CH system over $\mathbb{F}$.
Let the sequence $(\{\theta_i\}^d_{i=0}, \{\theta^*_i\}^d_{i=0}, \{\phi_i\}^d_{i=1})$ denote the parameter array of $\Phi$. 
Recall the scalars $\{\vartheta_i\}^{d+1}_{i=0}$ from \eqref{vartheta}.
Then $\vartheta_1\ne\vartheta_d$.
\end{lemma}
\begin{proof}
Since $\Phi$ is recurrent, there exists $\beta \in \mathbb{F}$ such that each of $\{\theta_i\}^d_{i=0}$, $\{\theta^*_i\}^d_{i=0}$, $\{\vartheta_i\}^{d+1}_{i=0}$ is $\beta$-recurrent.
We have $E_0A^*E_d \ne 0$ by Definition \ref{Def:CHS}(iv).
The result follows from these comments and Lemma \ref{lem: vanishing prod terms(2)}(iii).
\end{proof}

\begin{theorem}\label{thm:four families}
Let $\Phi=(A; \{E_i\}^d_{i=0}; A^*; \{E^*_i\}^d_{i=0})$ denote a recurrent CH system over $\mathbb{F}$.
Then $\Phi$ is isomorphic to a member of one of the four families in Examples {\rm\ref{ex:beta not 2, -2}}--{\rm\ref{ex:beta = 0}}.
%Every recurrent CH system is isomorphic to a member of one of the four families in Examples {\rm\ref{ex:beta not 2, -2}}--{\rm\ref{ex:beta = 0}}.
\end{theorem}
\begin{proof}
Let the sequence $(\{\theta_i\}^d_{i=0}, \{\theta^*_i\}^d_{i=0}, \{\phi_i\}^d_{i=1})$ denote the parameter array of $\Phi$. 
Recall the scalars $\{\vartheta_i\}^{d+1}_{i=0}$ from \eqref{vartheta}.
Since $\Phi$ is recurrent, there exists $\beta \in \mathbb{F}$ such that each of $\{\theta_i\}^d_{i=0}$, $\{\theta^*_i\}^d_{i=0}$, $\{\vartheta_i\}^{d+1}_{i=0}$ is $\beta$-recurrent.
We have $\vartheta_1 \ne \vartheta_d$ by Lemma \ref{lem:vth1=vthd}.
Consider the $\Phi$-split sequence $\{\phi_i\}^d_{i=1}$.
By \eqref{vartheta},
\begin{equation}\label{eq:phi-vartheta}
	\phi_i = \vartheta_i + (\theta^*_i-\theta^*_0)(\theta_{d-i+1}-\theta_0)  \qquad (1\leq i \leq d).
\end{equation}
We show that $\Phi$ is isomorphic to a CH system from one of the Examples {\rm\ref{ex:beta not 2, -2}}--{\rm\ref{ex:beta = 0}}.
We divide the argument into the following four cases: (i) $\beta\ne \pm2$; (ii) $\beta=2$ and $\operatorname{Char}(\mathbb{F})\ne 2$; (iii) $\beta=-2$ and $\operatorname{Char}(\mathbb{F})\ne 2$; (iv) $\beta=0$ and $\operatorname{Char}(\mathbb{F}) = 2$.

\medskip \noindent 
\textbf{Case} (i): $\beta\ne \pm2$.

\noindent
Pick $0\ne q \in \overline{\mathbb{F}}$ such that $q+q^{-1}=\beta$.
By Lemma \ref{lem:rec seq q-form}(i) there exist scalars $a,b,c, a^*, b^*, c^* \in \overline{\mathbb{F}}$ such that 
\begin{align}
	\theta_i & = a + bq^i + cq^{-i},  \label{eq(1):2nd thm pf}\\
	\theta^*_i & = a^* + b^*q^i + c^*q^{-i} \label{eq(2):2nd thm pf}
\end{align}
for $0 \leq i \leq d$.
We claim that $q^{d+1}=1$  and $q^i\ne 1$ for $1\leq i \leq d$.
Since the sequence $\{\vartheta_i\}^{d+1}_{i=0}$ is $\beta$-recurrent, by Lemma \ref{lem:rec seq q-form}(i) there exist scalars $x, y,z \in \overline{\mathbb{F}}$ such that
\begin{equation}\label{eq:vartheta_i q-form}
	\vartheta_i = x + yq^i + zq^{-i}
\end{equation}
for $0\leq i \leq d+1$. 
Using $\vartheta_0=0$ and $\vartheta_{d+1}=0$, we have $x+y+z=0$ and
\begin{equation}\label{eq(1):beta ne pm 2}
	x+yq^{d+1}+zq^{-d-1}=0.
\end{equation}
Eliminate $x$ in \eqref{eq(1):beta ne pm 2} using $x=-y-z$ and evaluate the result to get
\begin{equation}
	(1-q^{d+1})(y-zq^{-d-1})=0.
\end{equation}
Suppose that $q^{d+1}\ne 1$.
Then $y=zq^{-d-1}$. 
Using this equation, eliminate $y$ in \eqref{eq:vartheta_i q-form} to obtain
\begin{equation}\label{eq(2):beta ne pm 2}
	\vartheta_i = x+zq^{i-d-1} + zq^{-i} = \vartheta_{d-i+1}
\end{equation}
for $0 \leq i \leq d+1$.
In particular, for $i=1$ we have $\vartheta_1=\vartheta_d$, a contradiction.
Therefore, $q^{d+1}=1$.
Next, we show that $q^i\ne 1$ for $1\leq i \leq d$.
Recall the eigenvalue sequence $\{\theta_i\}^d_{i=0}$ of $\Phi$.
If $q^i=1$ for some $1\leq i \leq d$, then $\theta_i=\theta_0$ by \eqref{eq(1):2nd thm pf}; this is a contradiction since the scalars $\{\theta_i\}^d_{i=0}$ are mutually distinct. 
Thus, $q^i \ne 1$ for $1\leq i \leq d$.
We have proved the claim.

Next, we discuss the scalars $\{\phi_i\}^d_{i=1}$.
Evaluate the right-hand side of \eqref{eq:phi-vartheta} using \eqref{eq(1):2nd thm pf}--\eqref{eq:vartheta_i q-form} to obtain
\begin{equation}\label{eq(3):2nd thm pf}
	\phi_i  = (q^i-1)(y-zq^{-i}) + (q^i-1)(q^{-i}-1)(b-cq^{i})(b^*-c^*q^{-i})
\end{equation}
for $1\leq i \leq d$.
Comparing \eqref{eq:theta(i)}, \eqref{eq:theta*(i)}, \eqref{eq:phi(i)} with \eqref{eq(1):2nd thm pf}, \eqref{eq(2):2nd thm pf}, \eqref{eq(3):2nd thm pf}, respectively, and using Lemma \ref{iso_CHsystem}, we find that $\Phi$ is isomorphic to the CH system shown in \eqref{CHS_beta ne pm2,-2} of Example \ref{ex:beta not 2, -2}.

\medskip \noindent 
\textbf{Case} (ii): $\beta=2$ and $\operatorname{Char}(\mathbb{F})\ne 2$.

\noindent
By Lemma \ref{lem:rec seq q-form}(ii) there exist scalars $a,b,c,a^*, b^*,c^* \in \mathbb{F}$ such that 
\begin{align}
	\theta_i & = a + bi + ci(i-1)/2, \label{eq:theta(i)beta=2}\\
	\theta^*_i & = a^* + b^*i + c^*i(i-1)/2 \label{eq:theta*(i)beta=2}
\end{align}
for $0\leq i \leq d$.
We claim that $\operatorname{Char}(\mathbb{F})=d+1$.
Since the sequence $\{\vartheta_i\}^{d+1}_{i=0}$ is $\beta$-recurrent, by Lemma \ref{lem:rec seq q-form}(ii) there exist scalars $x, y,z \in \mathbb{F}$ such that
\begin{equation}\label{eq:vartheta_i; b=2}
	\vartheta_i = x + y i + z{i}(i-1)/2
\end{equation}
for $0\leq i \leq d+1$.
Using $\vartheta_0=0$ and $\vartheta_{d+1}=0$, we have $x=0$ and
\begin{equation}
	(d+1)(y+zd/2)=0.
\end{equation}
Suppose that $d+1\ne 0$ in $\mathbb{F}$.
Then $y=-zd/2$. 
Using this equation, eliminate $y$ in \eqref{eq:vartheta_i; b=2} to obtain
\begin{equation}
	\vartheta_i =  \frac{(i-d-1)i}{2}z = \vartheta_{d-i+1}
\end{equation}
for $0 \leq i \leq d+1$.
In particular, for $i=1$ we have $\vartheta_1=\vartheta_d$, a contradiction.
Therefore, $d+1=0$ in $\mathbb{F}$.
Next, we show that $i \ne 0$ in $\mathbb{F}$ for $1\leq i \leq d$.
Let $i$ be given.
If $i=0$ in $\mathbb{F}$, then $\theta_i=\theta_0$ by \eqref{eq:theta(i)beta=2}; 
this is a contradiction since the scalars $\{\theta_i\}^d_{i=0}$ are mutually distinct. 
Thus, $i \ne 0$ in $\mathbb{F}$.
We have proved the claim.

Next, we discuss the scalars $\{\phi_i\}^d_{i=1}$.
Evaluate the right-hand side of \eqref{eq:phi-vartheta} using \eqref{eq:theta(i)beta=2}--\eqref{eq:vartheta_i; b=2} to obtain
\begin{equation}\label{eq:phi formula b=2}
	\phi_i  = i\big(y+z(i-1)/2\big) - i^2\big(b+c(d-i)/2\big)\big(b^*+c^*(i-1)/2\big)
\end{equation} 
for $1\leq i \leq d$.
Comparing \eqref{eq:th cl form, be=2}, \eqref{eq:th* cl form, be=2}, \eqref{eq:phi cl form, be=2} with \eqref{eq:theta(i)beta=2}, \eqref{eq:theta*(i)beta=2}, \eqref{eq:phi formula b=2}, respectively, and using Lemma \ref{iso_CHsystem}, we find that $\Phi$ is isomorphic to the CH system shown in \eqref{CHS_beta=2} of Example \ref{ex:beta = 2}.

\medskip \noindent 
\textbf{Case} (iii): $\beta=-2$ and $\operatorname{Char}(\mathbb{F})\ne 2$.

\noindent
By Lemma \ref{lem:rec seq q-form}(iii) there exist scalars $a,b,c,a^*, b^*,c^* \in \mathbb{F}$ such that 
\begin{align}
	\theta_i & = a + b(-1)^i + c{i}(-1)^i, \label{eq:theta(i)beta=-2}\\
	\theta^*_i & = a^* + b^*(-1)^i + c^*{i}(-1)^i  \label{eq:theta*(i)beta=-2}
\end{align}
for $0\leq i \leq d$.
We claim that $d$ is odd, $d\geq 5$, and $\operatorname{Char}(\mathbb{F})=(d+1)/2$.
We first show that $d$ is odd.
Since the sequence $\{\vartheta_i\}^{d+1}_{i=0}$ is $\beta$-recurrent, by Lemma \ref{lem:rec seq q-form}(iii) there exist scalars $x, y, z \in \mathbb{F}$ such that
\begin{equation}\label{eq:vartheta_i; b=-2}
	\vartheta_i = x + y(-1)^i + z{i}(-1)^i
\end{equation}
for $0\leq i \leq d+1$.
Since $\vartheta_0=0$ and $\vartheta_{d+1}=0$, we have $x+y=0$ and
\begin{equation}\label{eq(1):be=-2}
	x+y(-1)^{d+1}+z(d+1)(-1)^{d+1}=0.
\end{equation}
Eliminate $x$ in \eqref{eq(1):be=-2} using $x=-y$ and evaluate the result to get
\begin{equation}\label{lem:eq(1)beta=-2}
	y\big((-1)^{d+1}-1\big) + z(d+1)(-1)^{d+1}=0.
\end{equation}
Suppose that $d$ is even. 
Simplify the equation in \eqref{lem:eq(1)beta=-2} to get $2y=-(d+1)z$. 
Using this equation along with $x+y=0$,  eliminate $x$ and $y$ in \eqref{eq:vartheta_i; b=-2} and simplify the result to obtain
\begin{equation}
	\vartheta_i = z\left( \frac{d+1}{2} - \frac{d+1}{2}(-1)^i + i (-1)^i \right) =\vartheta_{d-i+1}
\end{equation}
for $0\leq i \leq d+1$.
In particular, for $i=1$ we have $\vartheta_1= \vartheta_d$, a contradiction. 
Therefore, $d$ is odd.
Next, we show that $\operatorname{Char}(\mathbb{F})=(d+1)/2$.
Abbreviate $p=\operatorname{Char}(\mathbb{F})$.
Recall that $p$ is zero or a prime number.
Since $d$ is odd, we simplify the equation \eqref{lem:eq(1)beta=-2} to get $z(d+1)=0$. 
If $d+1\ne 0$ in $\mathbb{F}$, then $z=0$. 
So, by \eqref{eq:vartheta_i; b=-2} it follows that $\vartheta_i = -y+y(-1)^i$ for $0\leq i \leq d+1$.
However, this is a contradiction as $\vartheta_1=-2y = \vartheta_d$.
Thus, $d+1=0$ in $\mathbb{F}$. 
This implies that $p$ is a prime factor of $d+1$.
Next, we show that $(d+1)/2\leq p$.
Suppose not.
Then $2p\leq d-1$.
Setting $i=2p$ in \eqref{eq:theta(i)beta=-2}, we have $\theta_i=\theta_0$, a contradiction.
We have shown that $(d+1)/2\leq p$ and $p$ divides $d+1$. 
Therefore, either $p=(d+1)/2$ or $p=d+1$.
We have $p$ is prime and $d$ is odd and $d\geq 3$, so $p\ne d+1$.
Therefore, $p=(d+1)/2$.
Lastly, we show that $d\geq 5$. 
If $d=3$, then $p=2$. 
But, $p\ne 2$ by assumption. 
Therefore, $d\geq 5$.
We have proved the claim.

Next, we discuss the scalars $\{\phi_i\}^d_{i=1}$.
Evaluate the right-hand side of \eqref{eq:phi-vartheta} using \eqref{eq:theta(i)beta=-2}--\eqref{eq:vartheta_i; b=-2} to obtain
\begin{equation}\label{eq:phi formula b=-2}
	\phi_i  =  y\Big((-1)^i-1\Big)+zi(-1)^{i} 
		+ \Big(b\big((-1)^i-1\big)-ci(-1)^i\Big)\Big(b^*\big((-1)^i-1\big)+c^*i(-1)^i\Big)
\end{equation}
for $1\leq i \leq d$.
Comparing \eqref{eq:th cl form, be=-2}, \eqref{eq:th* cl form, be=-2}, \eqref{eq:phi cl form, be=-2} with \eqref{eq:theta(i)beta=-2}, \eqref{eq:theta*(i)beta=-2}, \eqref{eq:phi formula b=-2}, respectively, and using Lemma \ref{iso_CHsystem}, we find that $\Phi$ is isomorphic to the CH system shown in \eqref{CHS_beta=-2} of Example \ref{ex:beta = -2}.

\medskip \noindent 
\textbf{Case} (iv): $\beta=0$ and $\operatorname{Char}(\mathbb{F}) = 2$.

\noindent
By Lemma \ref{lem:rec seq q-form}(iv) there exist scalars $a,b,c,a^*, b^*,c^* \in \mathbb{F}$ such that 
\begin{align}
	\theta_i & = a + bi + c\binom{i}{2}, \label{eq:theta(i)beta=0}\\
	\theta^*_i & = a^* + b^*i + c^*\binom{i}{2} \label{eq:theta*(i)beta=0}
\end{align}
for $0\leq i \leq d$.
We claim that $d=3$.
If $d\geq 4$, then $\theta_0=\theta_4$ by \eqref{eq:theta(i)beta=0}; 
this is a contradiction since the scalars $\{\theta_i\}^4_{i=0}$ are mutually distinct. 
Since $d\geq 3$, the claim follows.

Next, we discuss the scalars $\{\phi_i\}^3_{i=1}$.
Since the sequence $\{\vartheta_i\}^{3}_{i=0}$ is $\beta$-recurrent, by Lemma \ref{lem:rec seq q-form}(iv) there exist scalars $x, y, z \in \mathbb{F}$ such that
\begin{equation}\label{eq:vartheta_i; b=0}
	\vartheta_i = x + yi + z\binom{i}{2}
\end{equation}
for $0\leq i \leq 4$.
Since $\vartheta_0=0$, we have $x=0$.
Evaluate the right-hand side of \eqref{eq:phi-vartheta} using \eqref{eq:theta(i)beta=0}--\eqref{eq:vartheta_i; b=0} to obtain
\begin{equation}\label{eq:phi formula b=0}
	\phi_i  =  yi + z\binom{i}{2} + \left(bi+c\binom{i+1}{2}\right)\left(b^*i+c^*\binom{i}{2}\right)
\end{equation}
for $1\leq i \leq 3$.
Comparing \eqref{eq:th cl form, be=0}, \eqref{eq:th* cl form, be=0}, \eqref{eq:phi cl form, be=0} with \eqref{eq:theta(i)beta=0}, \eqref{eq:theta*(i)beta=0}, \eqref{eq:phi formula b=0}, respectively, and using Lemma \ref{iso_CHsystem}, we find that $\Phi$ is isomorphic to the CH system shown in \eqref{CHS_beta=0} of Example \ref{ex:beta = 0}.
The proof is complete.
\end{proof}

%\begin{lemma}
%Let $\Phi=(A; \{E_i\}^d_{i=0}; A^*; \{E^*_i\}^d_{i=0})$ be a recurrent CH system.
%Then there exists a sequence of scalars $\beta, \gamma, \gamma^*, \varrho, \varrho^*$ taken from $\mathbb{F}$ such that both the tridiagonal relations \eqref{TD1} and \eqref{TD2} hold.
%\end{lemma}
%\begin{proof}
%By Theorem \ref{thm:four families}, $\Phi$ belongs to one of the four families listed in Examples \ref{ex:beta not 2, -2}--\ref{ex:beta = 0}.
%By these examples, there exists $\beta \in \mathbb{F}$ such that each of sequences $\{\theta_i\}^d_{i=0}$, $\{\theta^*_i\}^d_{i=0}$, $\{\vartheta_i\}^{d+1}_{i=0}$ is $\beta$-recurrent.
%By Lemma \ref{lem:TD-beta rec}, the result follows.
%\end{proof}

We finish this section with a comment.
Pick $\beta\in\mathbb{F}$.
Let $\Phi=(A; \{E_i\}^d_{i=0}; A^*; \{E^*_i\}^d_{i=0})$ denote a $\beta$-recurrent CH system over $\mathbb{F}$.
Let $(\{\theta_i\}^d_{i=0}$, $\{\theta^*_i\}^d_{i=0}$, $\{\phi_i\}^d_{i=1})$ be the parameter array of $\Phi$.
Let the scalars $\{\vartheta_i\}^{d+1}_{i=0}$ be from \eqref{vartheta}.
Recall from Lemma \ref{lem:vth1=vthd} that $\vartheta_1\ne \vartheta_d$.
In the following lemma, we express $\vartheta_i$ $(1\leq i \leq d)$ as a linear combination of $\vartheta_1$ and $\vartheta_d$.

\begin{lemma}
With the above notation, for $1\leq i \leq d$ the following {\rm(i)}--{\rm(iv)} hold.
\begin{enumerate}[\normalfont(i)]
\item Assume $\beta\ne \pm2$. 
	Then
	\begin{equation}\label{eq: vth = l.c. vth1 vth2 case(i)}
		\vartheta_i  =  \frac{(q^i-1)(q^{d-i}-1)}{(q-1)(q^{d-1}-1)} \vartheta_1
		+  \frac{(q^{i-1}-1)(q^{d-i+1}-1)}{(q-1)(q^{d-1}-1)} \vartheta_d,
	\end{equation}
	where $q+q^{-1}=\beta$.
\item Assume $\beta = 2$ and $\operatorname{Char}(\mathbb{F})\ne 2$.
	Then
	\begin{equation}\label{eq: vth = l.c. vth1 vth2 case(ii)}
		\vartheta_i  =  \frac{i(d-i)}{d-1} \vartheta_1
		+  \frac{(i-1)(d-i+1)}{d-1} \vartheta_d.
	\end{equation}
\item Assume $\beta=- 2$, $\operatorname{Char}(\mathbb{F})\ne 2$, and $d$ odd.
	Then
	\begin{equation}\label{eq: vth = l.c. vth1 vth2 case(iii)}
		\vartheta_i  =  
		\begin{cases}
		\dfrac{i}{d-1} \vartheta_1 + \dfrac{d-i+1}{d-1} \vartheta_d & \qquad \text{if \quad  $i$ is even};  \vspace{0.2cm} \\ 
		\dfrac{d-i}{d-1} \vartheta_1 + \dfrac{i-1}{d-1} \vartheta_d & \qquad \text{if \quad  $i$ is odd}.
		\end{cases}
	\end{equation}
\item Assume $\beta=0$, $\operatorname{Char}(\mathbb{F})= 2$, and $d=3$. Then
	\begin{equation}
	\vartheta_2 = \vartheta_1+\vartheta_3.
	\end{equation}
\end{enumerate}
\end{lemma}
\begin{proof}
(i): Since the sequence $\{\vartheta_i\}^{d+1}_{i=0}$ is $\beta$-recurrent, by Lemma \ref{def:beta-rec}(i) there exist scalars $x,y,z \in \overline{\mathbb{F}}$ such that
\begin{equation}\label{eq: varth gen form(1)}
	\vartheta_i = x+yq^i + zq^{-i} \qquad \qquad (0\leq i \leq d+1).
\end{equation}
Since $\vartheta_0=0$, we have $x+y+z=0$. 
Using this equation, eliminate $x$ in \eqref{eq: varth gen form(1)} to get
\begin{equation}\label{eq: varth gen form(2)}
	\vartheta_i = (q^i-1)y+(q^{-i}-1)z \qquad \qquad (0\leq i \leq d+1).
\end{equation}
Note that $q^{d+1}=1$ as we saw below \eqref{eq(2):beta ne pm 2}.
%As we saw below \eqref{eq(2):beta ne pm 2}, $q^{d+1}=1$.
Evaluate \eqref{eq: varth gen form(2)} at $i=1$ and $i=d$, and solve these two equations for $y$ and $z$.
Simplify the result using $q^{d+1}=1$ to get
%To obtain (i), use \eqref{eq:varth fm case1} to get
%\begin{equation*}
%	\vartheta_1  = (q-1)y + (q^{-1}-1)z, \qquad 
%	\vartheta_d = (q^d-1)y + (q^{-d}-1)z.
%\end{equation*}
%As we saw below \eqref{eq(2):beta ne pm 2}, $q^{d+1}=1$.
%Solve the above equations for $y$ and $z$ and simplify the result using $q^{d+1}=1$ to get
\begin{align}
	y & = -\frac{\vartheta_1 + q^d\vartheta_d }{(q-1)(q^{d-1}-1)}, \label{pf:eq(1)varth fm}
		 \\
	z & = -\frac{q^d\vartheta_1 + \vartheta_d}{(q-1)(q^{d-1}-1)}.\label{pf:eq(2)varth fm}
\end{align}
Eliminate $y$ and $z$ in \eqref{eq: varth gen form(2)} using \eqref{pf:eq(1)varth fm} and \eqref{pf:eq(2)varth fm}, and simplify the result to get \eqref{eq: vth = l.c. vth1 vth2 case(i)}.\\
(ii)--(iv): Similar.
%The cases (ii)--(iv) are similarly obtained by using \eqref{eq:varth cl form, be=2}, \eqref{eq:varth cl form, be=-2}, \eqref{eq:varth cl form, be=0}, respectively.
\end{proof}

%%%%%%%%%%%%%%%%%%%%%%%%%%%%%%%%%%%%%%%%%%
%%%%%%%%%%%%%%%%%%%%%%%%%%%%%%%%%%%%%%%%%%
%%%%%%%%%%%%%%%%%%%%%%%%%%%%%%%%%%%%%%%%%%
%%%%%%%%%%%%%%%%%%%%%%%%%%%%%%%%%%%%%%%%%%
\section{Six bases for $V$} \label{sec:SB}%%%%
Throughout this section, let $\Phi=(A; \{E_i\}^d_{i=0}; A^*; \{E^*_i\}^d_{i=0})$ denote a CH system on $V$.
Let the sequence $(\{\theta_i\}^d_{i=0}, \{\theta^*_i\}^d_{i=0}, \{\phi_i\}^d_{i=1})$ denote the parameter array of $\Phi$. 
In this section, we obtain the following results.
We display six bases for $V$ that we find attractive.
We display the transition matrices between certain pairs of bases among the six.
We display the matrices that represent $A$ and $A^*$ with respect to the six bases.

We now describe the first of the six bases.
Consider the decomposition $\{E_iV\}^d_{i=0}$ of $V$.
Recall the nonzero vector $u^* \in E^*_0V$ from above line \eqref{phi-split basis}.
By \cite[Lemma 8.1]{2010GodjLAA}, the sequence $\{E_iu^*\}^d_{i=0}$ is a basis for $V$.
We say this basis is \emph{$\Phi$-standard}.
In the next two results, we recall some characterizations of the $\Phi$-standard basis.

\begin{lemma}[cf. {\cite[Proposition 8.9]{2010GodjLAA}}]\label{lem:SB char(1)}
Let $\{u_i\}^d_{i=0}$ denote a sequence of vectors in $V$, not all zero.
Then the sequence $\{u_i\}^d_{i=0}$ is a $\Phi$-standard basis for $V$ if and only if both 
\begin{enumerate}[\normalfont(i)]
	\item $u_i \in E_iV$ for $0\leq i \leq d$, 
	\item $\sum^d_{i=0} u_i \in E^*_0V$.
\end{enumerate}
\end{lemma}
\begin{proof}
Suppose that  $\{ u_i \}^d_{i=0}$ is a $\Phi$-standard basis for $V$.
Then there exists $0\ne u^* \in E^*_0V$ such that $u_i = E_iu^*$ for $0\leq i \leq d$.
This implies (i).
Moreover, $\sum^d_{i=0}u_i = \sum^d_{i=0} E_i u^* = Iu^*=u^* \in E^*_0V$, and thus (ii) holds.
Conversely, suppose that the sequence $\{u_i\}^d_{i=0}$ satisfies (i) and (ii).
Set $u^* = \sum^d_{i=0} u_i$. 
Then $u^*$ is nonzero and contained in $E^*_0V$.
From (i) we have $E_iu^* = E_i(\sum^d_{j=0} u_j) = u_i$ for $0\leq i \leq d$.
Therefore, $\{u_i\}^d_{i=0}$ is a $\Phi$-standard basis for $V$.
\end{proof}

\noindent
Let $B$ denote a matrix in $\MatF$ and let $\alpha$ denote a scalar in $\mathbb{F}$.
Then $B$ is said to have \emph{constant row sum $\alpha$} whenever $B_{i0}+ B_{i1} + \cdots + B_{id} = \alpha$ for $0\leq i \leq d$.

\begin{lemma}[cf. {\cite[Proposition 8.10]{2010GodjLAA}}]\label{lem:SB char(2)}
Let $\{u_i\}^d_{i=0}$ denote a basis for $V$.
Let $B$ (resp. $B^*$) denote the matrix representing $A$ (resp. $A^*$) with respect to $\{u_i\}^d_{i=0}$.
Then $\{u_i\}^d_{i=0}$ is a $\Phi$-standard basis for $V$ if and only if both 
\begin{enumerate}[\normalfont(i)]
	\item $B=\mathrm{diag}(\theta_0, \theta_1, \ldots, \theta_d)$,
	\item $B^*$ has constant row sum $\theta^*_0$.
\end{enumerate}
\end{lemma}
\begin{proof}
Recall that for $0\leq i \leq d$, $E_iV$ is the eigenspace of $A$ associated with $\theta_i$.
Therefore $u_i\in E_iV$ for $0\leq i \leq d$ if and only if $B=\mathrm{diag}(\theta_0, \theta_1, \ldots, \theta_d)$.
Next, we observe that
$$
	A^* \sum^d_{j=0}u_j 
	= \sum^d_{j=0}A^*u_j 
	= \sum^d_{j=0} \sum^d_{i=0} B_{ij}u_i
	= \sum^d_{i=0} \sum^d_{j=0} B_{ij}u_i	
	= \sum^d_{i=0} (B_{i0} + B_{i1} + \cdots + B_{id})u_i.
$$
From this, it follows that $\sum^d_{j=0}u_j \in E^*_0V$ if and only if $B^*$ has constant row sum $\theta^*_0$.
The result follows from these comments and Lemma \ref{lem:SB char(1)}.
\end{proof}

We have discussed the $\Phi$-standard basis for $V$.
We now discuss five more bases for $V$.
Recall the $\Phi$-split basis $\{v_i\}^d_{i=0}$ for $V$ from \eqref{phi-split basis}.
The sequence $\{v_{d-i}\}^d_{i=0}$ is a basis for $V$.
We call this basis the \emph{inverted $\Phi$-split basis}.
Recall the dual CH system $\Phi^*=(A^*, \{E^*_i\}^d_{i=0}, A, \{E_i\}^d_{i=0})$.
Recall the nonzero vector $u \in E_0V$ from above line \eqref{phi*-split basis}.
By definition, the sequence $\{E^*_i u\}^d_{i=0}$ is a $\Phi^*$-standard basis for $V$.
Recall the $\Phi^*$-split basis $\{v^*_i\}^d_{i=0}$ for $V$ from \eqref{phi*-split basis}.
We consider the inverted $\Phi^*$-split basis $\{v^*_{d-i}\}^d_{i=0}$ for $V$.
We have now described six bases for $V$. The six bases are shown in the table below.
{\renewcommand{\arraystretch}{1.2}
\begin{equation}\label{six bases}
\begin{tabular}{l|cl}
	\qquad \quad name & & basis \\
	\hline
	$\Phi$-standard basis 	&& $\{E_iu^*\}^d_{i=0}$ \\
	$\Phi$-split basis 		&& $\{v_i\}^d_{i=0}$ \\
	inverted $\Phi$-split basis && $\{v_{d-i}\}^d_{i=0}$ \\
	$\Phi^*$-standard basis	&& $\{E^*_iu\}^d_{i=0}$\\
	$\Phi^*$-split basis		&& $\{v^*_i\}^d_{i=0}$\\
	inverted $\Phi^*$-split basis	&& $\{v^*_{d-i}\}^d_{i=0}$
\end{tabular}
\end{equation}}

Our next goal is to describe the transition matrices between certain pairs of bases among the six.
First, we recall the notion of a transition matrix.
Suppose we are given two bases for $V$, denoted $\{x_i\}^d_{i=0}$ and $\{y_i\}^d_{i=0}$.
By the \emph{transition matrix} from $\{x_i\}^d_{i=0}$ to $\{y_i\}^d_{i=0}$, we mean the matrix $T\in\MatF$ such that $y_j=\sum^d_{i=0} T_{ij}x_i$ for $0\leq j \leq d$. 
Consider the following diagram:
\begin{equation}\label{6 bases diagram}
\begin{tikzcd}
	\{E_iu^*\}^d_{i=0} \arrow[r,-] & \{v_{d-i}\}^d_{i=0} \arrow[r,-] \arrow[d,-] & \{v_i^*\}^d_{i=0}  \arrow[d,-] &	 
	\\
	& \{v_{i}\}^d_{i=0} \arrow[r,-] & \{v_{d-i}^*\}^d_{i=0}   \arrow[r,-] & \{E^*_iu\}^d_{i=0}
\end{tikzcd}
\end{equation}
We will display the transition matrices between each pair of bases that are adjacent in the above diagram.

\begin{lemma}\label{lem:TM SB and invspB}
The transition matrix from $\{E_i u^*\}^d_{i=0}$ to $\{v_{d-i}\}^d_{i=0}$ is upper triangular with $(i,j)$-entry
\begin{equation}\label{eq: mat P}
	(\th_i-\th_{j+1})(\th_i-\th_{j+2}) \cdots (\th_i-\th_{d}) 
\end{equation}
for $0\leq i \leq j\leq d$.
Moreover, the transition matrix from $\{v_{d-i}\}^d_{i=0}$ to $\{E_i u^*\}^d_{i=0}$ is upper triangular with $(i,j)$-entry
\begin{equation}\label{eq: mat Pinv}
	\dfrac{1}{(\theta_j - \theta_{i})\cdots(\theta_j - \theta_{j-1})(\theta_j - \theta_{j+1})\cdots(\theta_j - \theta_{d})} 
\end{equation}
for $0\leq i \leq j\leq d$.
\end{lemma}
\begin{proof}
Let $P$ denote the upper triangular matrix in $\MatF$ with $(i,j)$-entry \eqref{eq: mat P} for $0\leq i \leq j \leq d$.
By \eqref{phi-split basis} we find that for $0\leq j \leq d$,
\begin{align*}
	v_{d-j}
	& = (A-\theta_{j+1}I)(A-\theta_{j+2}I) \cdots (A-\theta_{d}I)u^* \\
	& = \sum^d_{i=0}E_i(A-\theta_{j+1}I)(A-\theta_{j+2}I) \cdots (A-\theta_{d}I)u^*\\
	& = \sum^d_{i=0}(\theta_i-\theta_{j+1})(\theta_i-\theta_{j+2}) \cdots (\theta_i-\theta_{d})E_iu^*\\
	& = \sum^d_{i=0}P_{ij}E_iu^*.
\end{align*}
By these comments, $P$ is the transition matrix from $\{E_i u^*\}^d_{i=0}$ to $\{v_{d-i}\}^d_{i=0}$.

To get the second assertion, we find the inverse of $P$.
Since $P$ is upper triangular, the inverse of $P$ is also upper triangular.
Let $H$ denote the upper triangular matrix in $\MatF$ with $(i,j)$-entry \eqref{eq: mat Pinv} for $0\leq i \leq j \leq d$.
We claim that $H$ is the inverse of $P$.
To this end, it suffices to show that $PH=I$.
Observe that $PH$ is upper triangular since $P$ and $H$ are upper triangular.
Consider the $(i,j)$-entry of $PH$ for $0\leq i \leq j \leq d$.
If $i=j$, then $(PH)_{ii}=P_{ii}H_{ii}=1$ by \eqref{eq: mat P} and \eqref{eq: mat Pinv}.
Next, assume $i<j$.
We show that $(PH)_{ij}=0$. 
Since $\theta_0, \theta_1, \ldots, \theta_d$ are mutually distinct, it suffices to show that $(\theta_i-\theta_j)(PH)_{ij}=0$. 
Since $P_{i\ell}=0$ for $0\leq \ell \leq i-1$ by \eqref{eq: mat P} and $H_{\ell j}=0$ for $j+1\leq \ell \leq d$ by \eqref{eq: mat Pinv}, we have
\begin{equation*}
	(PH)_{ij} = \sum^{i-1}_{\ell=0} P_{i\ell}H_{\ell j} + \sum^j_{\ell=i} P_{i\ell}H_{\ell j} + \sum^d_{\ell=j+1} P_{i\ell}H_{\ell j}=\sum^j_{\ell =i} P_{i\ell}H_{\ell j}.
\end{equation*}
Therefore,
\begin{align*}
	(\theta_i-\theta_j)(PH)_{ij}
	& = \sum^{j}_{\ell=i} P_{i\ell}H_{\ell j}(\theta_i-\theta_j)\\
	& = \sum^{j}_{\ell=i} P_{i\ell}H_{\ell j}(\theta_i-\theta_\ell + \theta_\ell - \theta_j)\\
	& = \sum^{j}_{\ell=i} P_{i\ell}(\theta_i-\theta_\ell)H_{\ell j} - \sum^{j}_{\ell=i} P_{i\ell}H_{\ell j}(\theta_j-\theta_\ell)\\
	& = \sum^{j}_{\ell=i+1} P_{i,\ell-1}H_{\ell j} - \sum^{j-1}_{\ell=i} P_{i\ell}H_{\ell+1,j}.
\end{align*}
We notice that the last two sums are the same, so it follows that $(\theta_i-\theta_j)(PH)_{ij}=0$.
Combining all our above comments, we find that $H$ is the inverse of $P$.
Therefore, $H$ is the transition matrix from $\{v_{d-i}\}^d_{i=0}$ to $\{E_i u^*\}^d_{i=0}$.
\end{proof}

\noindent
Recall the scalar $\varepsilon^*$ from Definition \ref{notation epsilon}.

\begin{lemma}\label{lem:TM invSpB and dualSpB}
The transition matrix from $\{v_{d-i}\}^d_{i=0}$ to $\{v^*_i\}^d_{i=0}$ is diagonal with $(i,i)$-entry
	\begin{equation*}
	\frac{\varepsilon^*  (\theta^*_0-\theta^*_{1})(\theta^*_0-\theta^*_{2}) \cdots (\theta^*_0 - \theta^*_{d})}{\phi_{1}\phi_{2} \cdots \phi_{d-i}}
	\end{equation*}
for $0\leq i \leq d$.
Moreover, the transition matrix from $\{v^*_i\}^d_{i=0}$ to $\{v_{d-i}\}^d_{i=0}$ is diagonal with $(i,i)$-entry
	\begin{equation*}
	\frac{\phi_{1}\phi_{2} \cdots \phi_{d-i}}{\varepsilon^*  (\theta^*_0-\theta^*_{1})(\theta^*_0-\theta^*_{2}) \cdots (\theta^*_0 - \theta^*_{d})}
	\end{equation*}
for $0\leq i \leq d$.
\end{lemma}
\begin{proof}
By Proposition \ref{prop:v v*}.
\end{proof}

\noindent
Let $Z$ denote the matrix in $\MatF$ with $(i,j)$-entry
\begin{equation}\label{eq: Z}
	Z_{ij} = \begin{cases}
	1 & \text{if} \quad i+j=d,\\
	0 & \text{if} \quad i+j \ne d
	\end{cases}
	\qquad \qquad (0\leq i,j \leq d).
\end{equation}
Observe that $Z^2=I$.

\begin{lemma}\label{lem:TM spB and invspB}
The transition matrix from $\{v_i\}^d_{i=0}$ to $\{v_{d-i}\}^d_{i=0}$ is equal to $Z$.
Moreover, the transition matrix from $\{v_{d-i}\}^d_{i=0}$ to $\{v_{i}\}^d_{i=0}$ is equal to $Z$.
\end{lemma}
\begin{proof}
Immediate from \eqref{eq: Z}.
\end{proof}

\begin{lemma}\label{lem:dual TM spB and dual invspB}
The transition matrix from $\{v^*_i\}^d_{i=0}$ to $\{v^*_{d-i}\}^d_{i=0}$ is equal to $Z$.
Moreover, the transition matrix from $\{v^*_{d-i}\}^d_{i=0}$ to $\{v^*_{i}\}^d_{i=0}$ is equal to $Z$.
\end{lemma}
\begin{proof}
Immediate from \eqref{eq: Z}.
\end{proof}

\noindent
Recall the scalar $\varepsilon$ from Definition \ref{notation epsilon}.

\begin{lemma}\label{lem:TM dual invSpB and SpB}
The transition matrix from $\{v^*_{d-i}\}^d_{i=0}$ to $\{v_i\}^d_{i=0}$ is diagonal with $(i,i)$-entry
	\begin{equation*}
	\frac{\varepsilon  (\theta_0-\theta_{1})(\theta_0-\theta_{2}) \cdots (\theta_0 - \theta_{d})}{\phi_{i+1}\phi_{i+2} \cdots \phi_{d}}
	\end{equation*}
for $0\leq i \leq d$. Moreover, the transition matrix from $\{v_i\}^d_{i=0}$ to $\{v^*_{d-i}\}^d_{i=0}$ is diagonal with $(i,i)$-entry
	\begin{equation*}
	\frac{\phi_{i+1}\phi_{i+2} \cdots \phi_{d}}{\varepsilon  (\theta_0-\theta_{1})(\theta_0-\theta_{2}) \cdots (\theta_0 - \theta_{d})}
	\end{equation*}
for $0\leq i \leq d$.
\end{lemma}
\begin{proof}
By Proposition \ref{prop:v v*}.
\end{proof}

\begin{lemma}\label{lem:dual TM SB and invspB}
The transition matrix from $\{E^*_i u\}^d_{i=0}$ to $\{v^*_{d-i}\}^d_{i=0}$ is upper triangular with $(i,j)$-entry
\begin{equation*}
	(\th^*_i-\th^*_{j+1})(\th^*_i-\th^*_{j+2})\cdots (\th^*_i-\th^*_{d})
\end{equation*}
for $0\leq i \leq j \leq d$.
Moreover, the transition matrix from $\{v^*_{d-i}\}^d_{i=0}$ to $\{E^*_i u\}^d_{i=0}$ is upper triangular with $(i,j)$-entry
\begin{equation*}
	\dfrac{1}{(\theta^*_j - \theta^*_{i})\cdots(\theta^*_j - \theta^*_{j-1})(\theta^*_j - \theta^*_{j+1})\cdots(\theta^*_j - \theta^*_{d})}
\end{equation*}
for $0\leq i \leq j \leq d$.
\end{lemma}
\begin{proof}
Apply Lemma \ref{lem:TM SB and invspB} to $\Phi^*$.
\end{proof}

\noindent
We have a comment. 
In Lemmas \ref{lem:TM SB and invspB}--\ref{lem:dual TM SB and invspB}, we found  the transition matrices between any pair of bases that are adjacent in the diagram \eqref{6 bases diagram}.
Using these matrices and linear algebra, one can compute the transition matrix between any pair of bases among the six bases in \eqref{six bases}.

We now describe the matrices representing $A$ and $A^*$ with respect to each basis in \eqref{six bases}.
Recall from Proposition \ref{prop: CHpair mat split} that the matrices representing $A$ and $A^*$ with respect to the $\Phi$-split basis $\{v_i\}^d_{i=0}$ are given in \eqref{ALP split mat}.
Similarly, the matrices representing $A$ and $A^*$ with respect to the $\Phi^*$-split basis $\{v^*_i\}^d_{i=0}$ are given in \eqref{ALP dual split mat}.

\begin{lemma}\label{lem:mat rep invSB}
The matrices representing $A$ and $A^*$ with respect to the inverted $\Phi$-split basis $\{v_{d-i}\}^d_{i=0}$ are 
\begin{equation*}\label{CHP dual split mat}
	A: \quad \begin{pmatrix}
	\theta_0 & 1 & & & & \mathbf{0} \\
	 & \theta_{1} & 1 \\
	 & & \theta_{2} & \cdot \\
	 & & & \cdot & \cdot \\
	 & & & & \cdot & 1\\ 
	\mathbf{0} & & & & & \theta_d
	\end{pmatrix},
	\qquad \qquad 
	A^*: 	\quad \begin{pmatrix}
	\theta^*_d &  & & & & \mathbf{0}\\
	\phi_d & \theta^*_{d-1} &  & \\
	 & \phi_{d-1}& \theta^*_{d-2} & \\
	 & & \cdot & \cdot & &\\
	 & & & \cdot & \cdot & \\
	\mathbf{0} & & & &\phi_1 & \theta^*_0
	\end{pmatrix}.
\end{equation*}
\end{lemma}
\begin{proof}
Conjugate each matrix in \eqref{ALP split mat} by the matrix $Z$ from \eqref{eq: Z}.
\end{proof}

\begin{lemma}
The matrices representing $A$ and $A^*$ with respect to the inverted $\Phi^*$-split basis $\{v^*_{d-i}\}^d_{i=0}$ are 
\begin{equation*}\label{CHP dual split mat}
	A: \quad \begin{pmatrix}
	\theta_d &  & & & & \mathbf{0}\\
	\phi_1 & \theta_{d-1} &  & \\
	 & \phi_{2}& \theta_{d-2} & \\
	 & & \cdot & \cdot & &\\
	 & & & \cdot & \cdot & \\
	\mathbf{0} & & & &\phi_d & \theta_0
	\end{pmatrix},
	\qquad \qquad 
	A^*: 	\quad \begin{pmatrix}
	\theta^*_0 & 1 & & & & \mathbf{0} \\
	 & \theta^*_{1} & 1 \\
	 & & \theta^*_{2} & \cdot \\
	 & & & \cdot & \cdot \\
	 & & & & \cdot & 1\\ 
	\mathbf{0} & & & & & \theta^*_d
	\end{pmatrix}.
\end{equation*}
\end{lemma}
\begin{proof}
Conjugate each matrix in \eqref{ALP dual split mat} by the matrix $Z$ from \eqref{eq: Z}.
\end{proof}

\begin{proposition}\label{prop:mat rep A A* stB}
The matrix representing $A$ with respect to the $\Phi$-standard basis $\{E_iu^*\}^d_{i=0}$ is
\begin{equation}\label{A:SB}
	\mathrm{diag}(\theta_0, \theta_1, \theta_2, \ldots, \theta_d).
\end{equation}
The matrix representing $A^*$ with respect to $\{E_iu^*\}^d_{i=0}$ is circular Hessenberg:
\begin{equation}\label{A*:SB}
	\begin{pmatrix}
	a^*_0 & b^*_0& & &  & \xi^*\\
	c^*_1 & a^*_1& b^*_1 & & &  \\
	& c^*_2 & a^*_2& \cdot & &  \\
	& & \cdot & \cdot & \cdot &  \\
	& & & \cdot & \cdot & b^*_{d-1} \\
	\mathbf{0} & & & & c^*_d & a^*_d  \\
	\end{pmatrix},
\end{equation}
where
\begin{equation}\label{eq: xi*}
	\xi^* = \th^*_0-a^*_0-b^*_0,
\end{equation}
and 
\begin{align}
	c^*_i &= \frac{(\th_i-\th_d)(\th_i-\th_{d-1})\cdots (\th_i-\th_{i+1})}{(\th_{i-1}-\th_d)(\th_{i-1}-\th_{d-1})\cdots (\th_{i-1}-\th_i)}\phi_{d-i+1} \qquad (1\leq i \leq d), \label{scalar:c*i} \\
	a^*_0 & = \th^*_{d} + \frac{\phi_{d}}{\th_0-\th_{1}}, \label{scalar:a*0}\\
	a^*_i & = \th^*_{d-i} + \frac{\phi_{d-i}}{\th_i-\th_{i+1}} + \frac{\phi_{d-i+1}}{\th_i-\th_{i-1}}\qquad (1\leq i \leq d-1), \label{scalar:a*i}\\	
	a^*_d & = \th^*_{0}  + \frac{\phi_{1}}{\th_d-\th_{d-1}},\label{scalar:a*d}\\
	b^*_0 & = \frac{(\th_0-\th_d)(\th_0-\th_{d-1})\cdots (\th_0-\th_{2})}{(\th_{1}-\th_d)(\th_{1}-\th_{d-1})\cdots (\th_{1}-\th_{2})} \times \left( \th^*_{d-1}-\th^*_{d} + \frac{\phi_{d-1}}{\th_0-\th_{2}} + \frac{\phi_{d}}{{\th_{1}-\th_0}} \right), \label{scalar:b*0}\\
		b^*_i & = \frac{(\th_i-\th_d)(\th_i-\th_{d-1})\cdots (\th_i-\th_{i+2})}{(\th_{i+1}-\th_d)(\th_{i+1}-\th_{d-1})\cdots (\th_{i+1}-\th_{i+2})} \nonumber\\
	& \quad \times \left( \th^*_{d-i-1}-\th^*_{d-i} + \frac{\phi_{d-i-1}}{\th_i-\th_{i+2}} + \frac{\phi_{d-i}}{\th_{i+1}-\th_i} + \frac{\phi_{d-i+1}}{\th_{i-1}-\th_{i+1}} \right) \qquad (1\leq i \leq d-2), \label{scalar:b*i} \\
	b^*_{d-1} & = \th^*_{0}-\th^*_{1} + \frac{\phi_{1}}{\th_{d}-\th_{d-1}} + \frac{\phi_{2}}{\th_{d-2}-\th_{d}}. \label{scalar:b*(d-1)}
\end{align}
\end{proposition}

\begin{proof}
The first assertion follows from Lemma \ref{lem:SB char(2)}(i).
For the second assertion, let $B^*$ denote the matrix representing $A^*$ with respect to $\{E_iu^*\}^d_{i=0}$.
By Definition \ref{Def:CHS}(v), the matrix $B^*$ has the circular Hessenberg form \eqref{A*:SB}.
By Lemma \ref{lem:SB char(2)}(ii), we have \eqref{eq: xi*}.
We now show \eqref{scalar:c*i}--\eqref{scalar:b*(d-1)}.
Recall the transition matrix from $\{E_i u^*\}^d_{i=0}$ to $\{v_{d-i}\}^d_{i=0}$ from Lemma \ref{lem:TM SB and invspB}.
We denote this matrix by $P$.
Recall the matrix representing $A^*$ with respect to $\{v_{d-i}\}^d_{i=0}$ from Lemma \ref{lem:mat rep invSB}.
We denote this matrix by $C$.
By linear algebra,
\begin{equation}\label{eq(1):BP=PC}
	B^* P =  P C.
\end{equation}
For $1 \leq i \leq d$, evaluate the $(i, i-1)$-entry of both sides of \eqref{eq(1):BP=PC} to find 
\begin{equation}\label{eq(2):pf A*}
	c^*_i P_{i-1, i-1}  = \phi_{d-i+1} P_{i,i}.
\end{equation}
Solve the equation \eqref{eq(2):pf A*} for $c^*_i$ and simplify the result to get \eqref{scalar:c*i}.
Next, for $0 \leq i \leq d$ evaluate the $(i,i)$-entry of both sides of \eqref{eq(1):BP=PC} to find 
\begin{equation}\label{eq(3):pf A*}
	a^*_i  P_{i,i} + c^*_i  P_{i-1,i} =  \th^*_{d-i}  P_{i, i} + \phi_{d-i} P_{i,i+1}.
\end{equation}
Solve the equation \eqref{eq(3):pf A*} for $a^*_i$ and simplify the result to get \eqref{scalar:a*0}--\eqref{scalar:a*d}.
Next, for $0 \leq i \leq d-1$ evaluate the $(i, i+1)$-entry of both sides of \eqref{eq(1):BP=PC} to find 
\begin{equation}\label{eq(4):pf A*}
	c^*_i P_{i-1,i+1} + a^*_i P_{i,i+1} + b^*_i  P_{i+1,i+1}  =  \th^*_{d-i-1}  P_{i,i+1} + \phi_{d-i-1} P_{i,i+2}.
\end{equation}
Solve the equation \eqref{eq(4):pf A*} for $b^*_i$ and simplify the result to get \eqref{scalar:b*0}--\eqref{scalar:b*(d-1)}.
The result follows.
\end{proof}

\begin{proposition}\label{prop:mat rep A A* dualstB}
The matrix representing $A^*$ with respect to the $\Phi^*$-standard basis $\{E^*_iu\}^d_{i=0}$ is
\begin{equation}
	\mathrm{diag}(\theta^*_0, \theta^*_1, \theta^*_2, \ldots, \theta^*_d).
\end{equation}
The matrix representing $A$ with respect to $\{E^*_iu\}^d_{i=0}$ is circular Hessenberg:
\begin{equation}\label{A:SB Phi*}
	\begin{pmatrix}
	a_0 & b_0 & & & \xi\\
	c_1 & a_1 & b_1 & \\
	& c_2 & a_2 & \ddots \\
	& & \ddots & \ddots & b_{d-1}\\
	{\bf 0}& & & c_d & a_d
	\end{pmatrix},
\end{equation}
where 
\begin{equation}\label{xi}
	\xi = \th_0-a_0-b_0,
\end{equation}
and 
\begin{align}
	c_i &= \frac{(\th^*_i-\th^*_d)(\th^*_i-\th^*_{d-1})\cdots (\th^*_i-\th^*_{i+1})}{(\th^*_{i-1}-\th^*_d)(\th^*_{i-1}-\th^*_{d-1})\cdots (\th^*_{i-1}-\th^*_i)}\phi_{i} \qquad (1\leq i \leq d), \label{scalar:ci}\\
	a_0 & = \th_{d} + \frac{\phi_{1}}{\th^*_0-\th^*_{1}},\label{scalar:a0}\\
	a_i & = \th_{d-i} + \frac{\phi_{i+1}}{\th^*_i-\th^*_{i+1}} + \frac{\phi_{i}}{\th^*_i-\th^*_{i-1}} \qquad (1\leq i \leq d-1),\label{scalar:ai}\\
	a_d & = \th_{0} + \frac{\phi_{d}}{\th^*_d-\th^*_{d-1}},\label{scalar:ad}\\
	b_0 & = \frac{(\th^*_0-\th^*_d)(\th^*_0-\th^*_{d-1})\cdots (\th^*_0-\th^*_{2})}{(\th^*_{1}-\th^*_d)(\th^*_{1}-\th^*_{d-1})\cdots (\th^*_{1}-\th^*_{2})} \times \left( \th_{d-1}-\th_{d} + \frac{\phi_{2}}{\th^*_0-\th^*_{2}} + \frac{\phi_{1}}{\th^*_{1}-\th^*_0}  \right) \label{scalar:b0}\\
	b_i & = \frac{(\th^*_i-\th^*_d)(\th^*_i-\th^*_{d-1})\cdots (\th^*_i-\th^*_{i+2})}{(\th^*_{i+1}-\th^*_d)(\th^*_{i+1}-\th^*_{d-1})\cdots (\th^*_{i+1}-\th^*_{i+2})} \nonumber\\
	& \quad \times \left( \th_{d-i-1}-\th_{d-i} + \frac{\phi_{i+2}}{\th^*_i-\th^*_{i+2}} + \frac{\phi_{i+1}}{\th^*_{i+1}-\th^*_i} + \frac{\phi_{i}}{\th^*_{i-1}-\th^*_{i+1}} \right) \qquad (1\leq i \leq d-2), \label{scalar:bi}\\	
	b_{d-1} & = \th_{0}-\th_{1} + \frac{\phi_{d}}{\th^*_{d}-\th^*_{d-1}} + \frac{\phi_{d-1}}{\th^*_{d-2}-\th^*_{d}}.\label{scalar:b(d-1)}
\end{align}
\end{proposition}
\begin{proof}
Apply Proposition \ref{prop:mat rep A A* stB} to $\Phi^*$.
\end{proof}

%%%%%%%%%%%%%%%%%%%%%%%%%%%%%%%%%%%%%%%%%%
%%%%%%%%%%%%%%%%%%%%%%%%%%%%%%%%%%%%%%%%%%
%%%%%%%%%%%%%%%%%%%%%%%%%%%%%%%%%%%%%%%%%%
\section{The scalars $\xi$ and $\xi^*$}\label{sec:scalar_xi}
Throughout this section, let $\Phi=(A; \{E_i\}^d_{i=0}; A^*; \{E^*_i\}^d_{i=0})$ denote a recurrent CH system over $\mathbb{F}$.
Let $(\{\theta_i\}^d_{i=0}, \{\theta^*_i\}^d_{i=0}, \{\phi_i\}^d_{i=1})$ denote the parameter array of $\Phi$.
In Propositions \ref{prop:mat rep A A* stB} and \ref{prop:mat rep A A* dualstB}, we encountered the scalars $\xi$ and $\xi^*$. 
Our goal for this section is to obtain a formula for $\xi$ and $\xi^*$.
We give two versions; these are Proposition \ref{formula: xi, xi*} and Corollary \ref{cor:xi xi*}.
\begin{lemma} \label{psi=1}
We have
\begin{equation}\label{scalar psi}
	\prod^{d-1}_{i=2} \frac{\theta_0 - \theta_{i+1}}{\theta_1-\theta_i}=1,
	\qquad \qquad
	\prod^{d-1}_{i=2} \frac{\theta^*_0 - \theta^*_{i+1}}{\theta^*_1-\theta^*_i}=1.
\end{equation}
\end{lemma}
\begin{proof}
We first prove the equation on the left in \eqref{scalar psi}.
In this equation, let $\psi$ denote the product on the left.
Since $\Phi$ is recurrent, there exists $\beta\in \mathbb{F}$ such that the sequence $\{\theta_i\}^d_{i=0}$ is $\beta$-recurrent.
We divide the argument into four cases: (i) $\beta\ne\pm2$, (ii) $\beta=2$ and $\operatorname{Char}(\mathbb{F})\ne 2$, (iii) $\beta=-2$ and $\operatorname{Char}(\mathbb{F})\ne 2$, (iv) $\beta=0$ and $\operatorname{Char}(\mathbb{F}) = 2$.
In case (i), choose $0\ne q \in \overline{\mathbb{F}}$ such that $q+q^{-1}=\beta$.
Apply Lemma \ref{lem:rec quotient}(i) to $\psi$ and evaluate the result to get
\begin{equation}\label{psi q-term}
	\psi = \frac{(1-q^d)(1-q^{d-1})}{q^{d-2}(1-q)(1-q^2)}.
\end{equation}
We saw below \eqref{eq(2):beta ne pm 2} that $q^{d+1}=1$.
Using this fact, simplify the right-hand side of \eqref{psi q-term} to get $\psi=1$, as desired.
We have now proved the argument for case (i).
The remaining cases (ii)--(iv) are proved in a similar fashion.
We have proved the equation on the left in \eqref{scalar psi}.
Apply this result to $\Phi^*$ to get the equation on the right in \eqref{scalar psi}.
\end{proof}

\begin{proposition}\label{formula: xi, xi*}
Recall the scalars $\xi$ from \eqref{A:SB Phi*} and $\xi^*$ from \eqref{A*:SB}.
Then both
\begin{align}
	\xi & = \frac{\phi_1 -\phi_d}{\theta^*_1-\theta^*_d} + \frac{ (\theta_1-\theta_0)(\theta^*_d-\theta^*_0) - (\theta_d-\theta_0)(\theta^*_1-\theta^*_0) }{\theta^*_1-\theta^*_d},
	\label{eq:xi}\\
	\xi^* & = \frac{\phi_d -\phi_1}{\theta_1-\theta_d} + \frac{ (\theta^*_1-\theta^*_0)(\theta_d-\theta_0) -(\theta^*_d-\theta^*_0)(\theta_1-\theta_0) }{\theta_1-\theta_d}.
	\label{eq:xi*}
\end{align}
\end{proposition}
\begin{proof}
We first prove \eqref{eq:xi}.
Evaluate $b_0$ in \eqref{scalar:b0} using the equation on the right in \eqref{scalar psi} to get
\begin{equation}\label{simple form b0}
	b_0 = \frac{\theta^*_0-\theta^*_2}{\theta^*_1-\theta^*_d} \left( \th_{d-1}-\th_{d} + \frac{\phi_{2}}{\th^*_0-\th^*_{2}} + \frac{\phi_{1}}{\th^*_{1}-\th^*_0}  \right).
\end{equation}
Recall from \eqref{xi} that $\xi= \theta_0-a_0-b_0$.
Eliminate $a_0$ and $b_0$ in this equation using \eqref{scalar:a0} and \eqref{simple form b0} to get
\begin{equation}\label{scalar:xi}
	\xi = \th_0-\th_d - \frac{\phi_1}{\th^*_0-\th^*_1} - \frac{\theta^*_0-\theta^*_2}{\theta^*_1-\theta^*_d} \left( \th_{d-1}-\th_d + \frac{\phi_{1}}{\th^*_1-\th^*_0} + \frac{\phi_2}{\th^*_0-\th^*_2} \right).
\end{equation}
Consider the quantity that is the right-hand side of \eqref{eq:xi} minus the right-hand side of \eqref{scalar:xi}.
Evaluate this quantity using the closed forms for the scalars $\{\theta_i\}^d_{i=0}$, $\{\theta^*_i\}^d_{i=0}$, $\{\phi_i\}^d_{i=1}$ that are presented in Examples \ref{ex:beta not 2, -2}--\ref{ex:beta = 0}.
For each example, the above quantity is zero.
We have proved \eqref{eq:xi}.
To prove \eqref{eq:xi*}, apply \eqref{eq:xi} to $\Phi^*$.
\end{proof}

\begin{corollary}\label{cor:xi xi*}
Recall the scalars $\{\vartheta_i\}^{d+1}_{i=0}$ from \eqref{vartheta}.
We have
\begin{equation}\label{xi and xi*}
	\xi = \frac{\vartheta_1-\vartheta_d}{\theta^*_1-\theta^*_d}, \qquad \qquad
	\xi^* = \frac{\vartheta_d-\vartheta_1}{\theta_1-\theta_d}.
\end{equation}
\end{corollary}
\begin{proof}
In \eqref{eq:xi} and \eqref{eq:xi*}, eliminate $\phi_1$ and $\phi_d$ using \eqref{vartheta} and simplify the result.
\end{proof}
\begin{note}
Referring to \eqref{xi and xi*}, $\xi$ and $\xi^*$ are nonzero by Lemma \ref{lem:vth1=vthd}.
\end{note}

\section{Appendix}\label{appendix}
In this appendix, we recall some formulas involving recurrent sequences that are used in the main body of the paper.
Throughout this appendix, let $\beta$ denote any scalar in $\mathbb{F}$ and let $\{\theta_i\}^d_{i=0}$ denote an arbitrary sequence of scalars taken from $\mathbb{F}$.

\begin{definition}[{\cite[Definition 8.2]{2001TerLAA}}]\label{def:beta-rec}
%Let $\beta$ denote a scalar in $\mathbb{F}$.
The sequence $\{\theta_i\}^d_{i=0}$ is said to be \emph{$\beta$-recurrent} whenever 
\begin{equation*}
	\theta_{i-2} - (\beta+1)\theta_{i-1}+(\beta+1)\theta_{i} - \theta_{i+1}=0
\end{equation*}
for $2\leq i \leq d-1$.
\end{definition}

%We begin by recalling how a $\beta$-recurrent sequence looks in a closed form.

\begin{lemma}[cf. {\cite[Lemma 9.2]{2001TerLAA}}]\label{lem:rec seq q-form}
%Let $\beta$ denote a scalar in $\mathbb{F}$.
Assume that the sequence $\{\theta_i\}^d_{i=0}$ is $\beta$-recurrent.
Then the following {\rm (i)}--{\rm(iv)} hold.
\begin{enumerate}[\normalfont(i)]
	\item Suppose that $\beta \ne \pm 2$ and choose $0 \ne q \in \overline{\mathbb{F}}$ such that $q+q^{-1}=\beta$. 
	Then there exist scalars $\alpha_1, \alpha_2, \alpha_3$ in $\overline{\mathbb{F}}$ such that
	\begin{equation*}
		\theta_i = \alpha_1 + \alpha_2q^i + \alpha_3q^{-i}, \qquad 0\leq i \leq d.
	\end{equation*}
	\item Suppose $\beta=2$ and $\operatorname{Char}(\mathbb{F})\ne 2$. 
	Then there exist scalars $\alpha_1, \alpha_2, \alpha_3$ in $\mathbb{F}$ such that
	\begin{equation*}
		\theta_i = \alpha_1 + \alpha_2 i + \alpha_3 i(i-1)/2, \qquad 0\leq i \leq d.
	\end{equation*}
	\item Suppose $\beta=-2$ and $\operatorname{Char}(\mathbb{F})\ne 2$.
	Then there exist scalars $\alpha_1, \alpha_2, \alpha_3$ in $\mathbb{F}$ such that
	\begin{equation*}
		\theta_i = \alpha_1 + \alpha_2 (-1)^i + \alpha_3 i(-1)^i, \qquad 0\leq i \leq d.
	\end{equation*}
	\item Suppose $\beta=0$ and $\operatorname{Char}(\mathbb{F})= 2$. Then there exists $\alpha_1, \alpha_2, \alpha_3$ in $\mathbb{F}$ such that
	\begin{equation*}
		\theta_i = \alpha_1 + \alpha_2 i + \alpha_3 \binom{i}{2}, \qquad 0\leq i \leq d, 
	\end{equation*}
	where we interpret the binomial coefficient as follows:
	\begin{equation*}%\label{eq:bin char=2}
	\binom{i}{2} = \begin{cases}
	0 &  \text{if } \quad i=0 \text{ or } i=1 \quad (\mathrm{mod}\ 4),\\
	1 & \text{if } \quad  i=2 \text{ or } i=3 \quad (\mathrm{mod}\ 4).	
	\end{cases}
	\end{equation*}
\end{enumerate}
\end{lemma}

\begin{lemma}[cf. {\cite[Lemma 9.4]{2001TerLAA}}]\label{lem:rec quotient}
%Let $\beta$ denote a scalar in $\mathbb{F}$. 
Assume that the scalars $\{\theta_i\}^d_{i=0}$ are mutually distinct and $\beta$-recurrent.
Pick any integers $i,j,r,s$ $(0\leq i,j,r,s\leq d)$ such that $i+j=r+s$, $r\ne s$.
Then {\rm(i)}--{\rm(iv)} hold below.
\begin{enumerate}[\normalfont(i)]
\item Suppose $\beta\ne \pm2$. 
	Then
	\begin{equation*}
		\frac{\theta_i-\theta_j}{\theta_r-\theta_s} = \frac{q^i-q^j}{q^r-q^s},
	\end{equation*}
	where $q+q^{-1}=\beta$.
\item Suppose $\beta = 2$ and $\operatorname{Char}(\mathbb{F})\ne 2$.
	Then
	\begin{equation*}
		\frac{\theta_i-\theta_j}{\theta_r-\theta_s} = \frac{i-j}{r-s}.
	\end{equation*}
\item Suppose $\beta = -2$ and $\operatorname{Char}(\mathbb{F})\ne 2$.
	Then
	\begin{equation*}
		\frac{\theta_i-\theta_j}{\theta_r-\theta_s} = 
		\begin{cases}
		(-1)^{i+r}(i-j)/(r-s) & \qquad \text{if \quad $i+j$ is even};\\
		(-1)^{i+r} & \qquad \text{if \quad $i+j$ is odd}.
		\end{cases}
	\end{equation*}
\item Suppose $\beta = 0$ and $\operatorname{Char}(\mathbb{F}) = 2$.
	Then
	\begin{equation*}
		\frac{\theta_i-\theta_j}{\theta_r-\theta_s} = 
		\begin{cases}
		0 & \qquad \text{if} \quad i=j;\\
		1 & \qquad \text{if} \quad i \ne j.
		\end{cases}
	\end{equation*}
\end{enumerate}
\end{lemma}

\section*{Acknowledgement}
The author expresses his deepest gratitude to Paul Terwilliger for many conversations about this paper.

%%%%%%%%%%%%%%%%%%%%%%%%%%%%%%%%%%%%%%%%%%%%

%\bigskip

%\bigskip
%\noindent
%Jae-Ho Lee\\
%Department of Mathematics and Statistics\\
%University of North Florida\\
%1 UNF Drive\\
%Jacksonville, Florida, 32224 USA\\
%Email: {jaeho.lee@unf.edu}\\


\begin{thebibliography}{00}%%%%%%%%%%%%%%%%%%%%%%%%%%%%%%

\bibitem{1984BI} {E. Bannai, T. Ito}, 
	Algebraic Combinatorics I: Association Schemes, Benjamin/Cummings, London, 1984.

\bibitem{2006Baseilhac} {P. Baseilhac},
	A family of tridiagonal pairs and related symmetric functions,
	J. Phys. A: Math. Gen. 39 (2006), 11773--11791.

\bibitem{2017BGV} {P. Baseilhac, A.M. Gainutdinov, T.T. Vu},
	Cyclic tridiagonal pairs, higher order Onsager algebras and orthogonal polynomials,
	Linear Algebra Appl. 522 (2017), 71--110.


\bibitem{2005BK} {P. Baseilhac, K. Koizumi},
	A new (in)finite dimensional algebra for quantum integrable models,
	Nucl. Phys. B, 720 (2005), pp. 325--347.

\bibitem{2010BS} {P. Baseilhac, K. Shigechi}, 
	A new current algebra and the reflection equation,
	Lett. Math. Phys. 92 (2010), no. 1, 47--65. 

\bibitem{2014BC} {S. Bockting-Conrad}, 
	Tridiagonal pairs of $q$-Racah type, the double lowering operator $\psi$, and the quantum algebra $U_q(\mathfrak{sl}_2)$,
	Linear Algebra Appl. 445 (2014), 256--279.

\bibitem{2011INT} {T. Ito, K. Nomura, P. Terwilliger},
	A classification of sharp tridiagonal pairs,
	Linear Algebra Appl., 435 (2011), pp. 1857--1884.


\bibitem{2001ITP} {T. Ito, K. Tanabe, P. Terwilliger},
	Some algebra related to $P$- and $Q$-polynomial association schemes, 
	in: Codes and Association Schemes (Piscataway NJ, 1999), Amer. Math. Soc., Providence RI, 2001, pp. 167--192.

\bibitem{2007IT(2)} {T. Ito, P. Terwilliger},
	The $q$-tetrahedron algebra and its finite dimensional irreducible modules,
	Comm. Algebra 35 (2007), no. 11, 3415--3439.

\bibitem{2007IT} {T. Ito, P. Terwilliger},
	Tridiagonal pairs and the quantum affine algebra $U_q(\widehat{\mathfrak{sl}}_2)$,
	Ramanujan J., 13 (2007), pp. 39--62.

\bibitem{2009IT} {T. Ito, P. Terwilliger},
	Tridiagonal pairs of $q$-Racah type,
	J. Algebra, 322 (2009), pp. 68--93.

\bibitem{2010IT} {T. Ito, P. Terwilliger}, 
	The augmented tridiagonal algebra, 
	Kyushu J. Math. 64 (2010), 81--144.

\bibitem{1982Leonard} {D. A. Leonard},
	Orthogonal polynomials, duality and association schemes,
	SIAM J. Math.Anal. 13 (1982), 656--663.

\bibitem{2009GodjLAA} {A. Godjali},
	Hessenberg pairs of linear transformations,
	Linear Algebra Appl. 431 (2009), 1579--1586.

\bibitem{2010GodjLAA} {A. Godjali},
	Thin Hessenberg pairs,
	Linear Algebra Appl. 432 (2010), 3231--3249.

\bibitem{1992TerJAC} {P. Terwilliger}, 
	The subconstituent algebra of an association scheme, 
	J. Algebraic Combin. 1 (1992), 363--388.

\bibitem{2001TerLAA} {P. Terwilliger},
	Two linear transformations each tridiagonal with respect to an eigenbasis
	of the other, 
	Linear Algebra Appl. 330 (2001), 149--203.

\bibitem{2002TerRMJ} {P. Terwilliger}, 
	Leonard pairs from 24 points of view, 
	Conference on Special Functions (Tempe, AZ, 2000),
	Rocky Mountain J. Math. 32 (2002), 827--888.

\bibitem{2004TerLAA} {P. Terwilliger}, 
	Leonard pairs and the $q$-Racah polynomials,
	Linear Algebra Appl. 387 (2004), 235--276.

\bibitem{2005TerJCAM} {P. Terwilliger}, 
	Two linear transformations each tridiagonal with respect to an eigenbasis of the other: comments on the split decomposition,
	J. Comput. Appl. Math. 178 (2005), 437--452.

\bibitem{2006TerLNM} {P. Terwilliger}, 
	An algebraic approach to the Askey scheme of orthogonal polynomials,
	Orthogonal polynomials and special functions, 255--330, Lecture Notes in Math., 1883,
Springer, Berlin, 2006.
	
\bibitem{2020TerNote} {P. Terwilliger}
	Notes on the Leonard system classification, 
	Graphs Combin. 37 (2021), no. 5, 1687--1748.

\end{thebibliography}
\end{document}